\documentclass[a4paper,reqno]{article}
\usepackage[english]{babel}
\usepackage{a4wide}
\usepackage{graphicx}
\usepackage{amsmath,amssymb,amsthm, amsgen,mathrsfs}
\usepackage{stmaryrd}
\usepackage[latin1]{inputenc}
\usepackage{listings}
\usepackage{color}
\usepackage[usenames,dvipsnames]{xcolor}
\usepackage{rotating}
\usepackage{hhline}
\usepackage{subfigure}
\usepackage{multirow}
\usepackage{enumerate}
\usepackage{enumitem}
\usepackage{dsfont}
\usepackage{tikz} 
\usepackage{url}
\usepackage{mathtools}
\usepackage{algorithm}
\usepackage{algorithmic}
\usepackage[normalem]{ulem}
\usepackage{dsfont}

\numberwithin{equation}{section}
\newtheorem{thm}{Theorem}[section]
\newtheorem{prop}[thm]{Proposition}
\newtheorem{lem}[thm]{Lemma}

\newtheorem{cor}[thm]{Corollary}
\newtheorem{defn}[thm]{Definition}

\theoremstyle{definition}
\newtheorem{rem}[thm]{Remark}

\newcommand{\R}{\mathbb R}

\newcommand{\N}{\mathbb N}
\renewcommand{\P}{\mathbb P}
\newcommand{\Q}{\mathbb Q}

\newcommand{\E}{\mathbb E}
\newcommand{\e}{\varepsilon}  
\newcommand{\ig}{[\![}
\newcommand{\id}{]\!]}

\newcommand{\Dopp}{D^{\operatorname{opp}}}
\newcommand{\Dsame}{D^{\operatorname{same}}}

\newcommand{\dist}{\operatorname{dist}}

\newcommand{\SB}{SqB_0}
\newcommand{\SC}{SC}

\newcommand{\dd}{{\rm d}}

\newcommand{\bK}{\mathbf K}

\newcommand{\bL}{\mathbf L}

\newcommand{\PP}{\mathbb{P}}

\newcommand{\cE}{\mathcal E}
\newcommand{\cF}{\mathcal F}
\newcommand{\cG}{\mathcal G}
\newcommand{\cH}{\mathcal H}

\newcommand{\cM}{\mathcal M}
\newcommand{\cN}{\mathcal N}

\newcommand{\cP}{\mathcal P}

\newcommand{\cS}{\mathcal S}

\newcommand{\indiq}{\hbox{\rm 1}{\hskip -2.8 pt}\hbox{\rm I}}
\newcommand{\1}{\indiq}

\newcommand{\bbrack}[1]{\left\llbracket {#1} \right\rrbracket }
\newcommand{\ov}[1]{\overline{#1}}

\long\def\comm#1{}

\usepackage[colorlinks=true, linkcolor=black, citecolor=black, urlcolor=black]{hyperref}

\title{Strong well-posedness of a singular SDE for signed Coulomb particles}
\author{Patrick van Meurs\thanks{Faculty of Mathematics and Physics, Kanazawa University, Kakuma, Kanazawa 920-1192, Japan.
e-mail: pjpvmeurs@staff.kanazawa-u.ac.jp} 
\
and Yoan Tardy\thanks{CMAP, \'Ecole Polytechnique, Route de Saclay, 92160 Palaiseau,
Paris. e-mail: yoan.tardy@polytechnique.edu}
}
\date{}

\begin{document} 

\maketitle

\comm{Additional comments that we want to keep for now, but which we remove later. The black text should be self-consistent when the orange comments are removed.}

\begin{abstract}
We consider an SDE system for signed Coulomb particles moving in $\mathbb R^2$. Due to the singular Coulomb interaction force, collisions between particles of opposite sign will happen in finite time. Upon collision, the colliding particles are removed from the system. Our main results are the existence and uniqueness of strong global solutions and the characterization of all possible collisions. The challenge of the proofs is to deal with the singularity of the interactions. We overcome this by using scaling invariance of the process and by putting together several tools from \cite{FournierTardy25} developed for the similar Keller--Segel particle system.
\end{abstract}

\tableofcontents

\section{Introduction}
\label{s:intro}

\subsection{The particle system}
\label{s:intro:SC}

The particle system is formally described by the following system of SDEs:
\begin{align} \label{SDE}
  (\SC) \qquad 
  \left\{ \begin{aligned}
    &X_t^i = X^i_0 +  B_t^i + \gamma \sum_{j =1 }^N b^i b^j \int_0^t f(X_s^i - X_s^j) \dd s, && t > 0, \ i = 1,\ldots,N \\
    &\text{+ collision rule},
    &&
  \end{aligned} 
  \right.
\end{align}
where we describe the `collision rule' further below, `$\SC$' stands for signed Coulomb particle system, $N \geq 1$ is the number of particles, the process $(X^i_t)_{t\ge 0}$ represents the position in $\R^2$ of particle $i \in \bbrack{1,N}$, $(B^1_t,\dots, B^N_t)_{t\ge 0}$ is a standard $2N$-dimensional Brownian motion, $\gamma > 0$ is a fixed parameter which determines the strength of the interactions, $b^i \in \{-1,1\}$ is a fixed parameter that represents the sign of particle $i$, and 
\begin{equation*}
  f : \R^2 \to \R^2, \qquad
  f(x) = \frac x{|x|^2}, \qquad f(0) := 0
\end{equation*}
is the Coulomb interaction force. Figure \ref{fig:X} shows a schematic of the particle configuration $X_t = (X_t^1, \ldots, X_t^N) \in (\R^2)^N$. We impose that the initial positions $( X^i_0 )_{i=1}^N$ are distinct. The choice $f(0) = 0$ is made for convenience to exclude self-interaction.

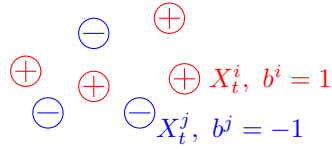
\begin{figure}[h]
\centering
\begin{tikzpicture}[scale=1]
    \def \r {0.2}
    \def \lgray {black!20!white}

   \foreach \x/\y in {0/0, 1/.9, -.9/.2, 1.2/.1}{
     \begin{scope}[shift={(\x,\y)}] 
     \filldraw[draw=red, fill=white] (0, 0) circle (\r); 
     \draw[red] (-.7*\r, 0) -- (.7*\r, 0);
     \draw[red] (0, -.7*\r) -- (0, .7*\r);
     \end{scope}
   }  

   \draw[red] (1.4,.1) node[right]{$X_t^i, \ b^i = 1$};  
   \draw[blue] (.7,-.55) node[below, right]{$X_t^j, \ b^j = -1$};  
   
   \foreach \i in {0,120,240}{
     \begin{scope}[rotate=\i] 
     \begin{scope}[shift={(0,.7)}, rotate=-\i] 
       \filldraw[draw=blue, fill=white] (0, 0) circle (\r); 
       \draw[blue] (-.7*\r, 0) -- (.7*\r, 0);
     \end{scope}
     \end{scope}
   }               
\end{tikzpicture}
\caption{Snapshot of $X_t = (X_t^1, \ldots, X_t^N) \in (\R^2)^N$ with signs $b = (b^1, \ldots, b^N)$.}
\label{fig:X}
\end{figure}

The system \eqref{SDE} was recently proposed and studied in \cite{VanMeursPeletierSlangen25}. A simple interpretation of \eqref{SDE} is that of electrically charged point particles that evolve by a `velocity = noise + force' law. In actual applications, \eqref{SDE} appears as a model of vortices or as a model for dislocations in metals. We refer to \cite[Section 1.1]{VanMeursPeletierSlangen25} for a detailed description.

One interesting observation that was mentioned in \cite{VanMeursPeletierSlangen25} is that the drift in \eqref{SDE} is the negative of the gradient of the Kirchhoff-Onsager functional given by
$$
    H(X) := -\gamma \sum_{1\le i < j\le N} b^ib^j\log (|X^i-X^j|).
    $$
Then, in vector notation in $(\R^2)^N$, the SDE system in \eqref{SDE} reads as 
\begin{equation} \label{SDE:vec-form}
  \dd X_t = \dd B_t + F(X_t)\dd t,   
\end{equation}
where $F(X) := -\nabla H(X)$ with
\[
  F^i(X) = -\nabla_{X^i} H(X) = \gamma \sum_{j =1 }^N b^i b^j f(X^i - X^j).
\]
The same functional $H$ is also studied in 
\cite{LebleSerfatyZeitouni17,BoursierSerfaty24ArXiv,BoursierSerfaty25ArXiv}. There, it is considered as the energy of the Gibbs measure, and the main interest is in the behavior as $N \to \infty$.

The main interesting feature of \eqref{SDE} is that particles can collide in finite time. Indeed, the occurrence of particle collisions is equivalent to the singularity of $f$ in the drift being reached. It is a matter of modeling how the collision is handled. In \cite{VanMeursPeletierSlangen25} no particular collision rule is applied; for $\gamma$ small enough the equation \eqref{SDE} can be interpreted in a weak sense that does not break down when particles collide. Then, particles instantaneously separate upon collision.

However, for the applications to vortices and dislocations mentioned above, it is more natural to \textit{remove} particles upon collision. Indeed, vortices and dislocations can be interpreted as topological defects that cancel each other out upon collision if their signs are opposite. This motivates us to apply the following collision rule: upon a collision (possibly between more than two particles), pairs of particles of opposite sign are instantaneously removed until all particles at the collision site have the same sign. Then, the remaining particles continue to evolve according to \eqref{SDE} (with $N$ reduced). We refer to this collision rule as the \textit{removal rule}. Since particles are always removed in pairs of opposite sign, the `net charge' given by $\sum_{i=1}^N b^i$ is conserved during collisions. 

A second motivation for removing particles upon collision is that for $\gamma \geq 1$ the drift in \eqref{SDE} is too strong for accommodating separation of particles after collision; see \cite[Section 6]{FournierJourdain17}. Therefore, \cite{VanMeursPeletierSlangen25} has not considered $(\SC)$ for $\gamma \geq 1$. If one would desire to keep the particles in the system, then colliding particles would have to stick together. This would mean that their combined force on a third particle sums to $0$ (because they are of opposite sign), and vice versa, the forces on both particles induced by a third particle cancel each other out. Hence, the collided particles would have to move as an independent Brownian motion, which is effectively the same as removing them from the system.

\subsection{Aim}
\label{s:intro:aim}

Henceforth, we consider \eqref{SDE} with the removal rule described above.
To the best of our knowledge there is no global-in-time well-posedness result on \eqref{SDE} that can handle collisions. 
The goal of this paper is to establish strong well-posedness of \eqref{SDE}, and to study the kind of collisions that may occur. The main challenges to overcome are that if a collision happens at $\tau$, then we need to prove that:
\begin{enumerate}[label=(\Alph*)]
  \item \label{proA} $X_{\tau-} := \lim_{t \nearrow \tau} X_t$ exists a.s., and 
  \item \label{proB} the particles in $X_\tau$ are distinct, where $X_\tau$ is the configuration obtained after applying the collision rule to $X_{\tau-}$.
\end{enumerate}
Property \ref{proB} limits the type of collisions that can occur; it only allows for collisions for which $|n^+ - n^-| \leq 1$, where $n^+$ ($n^-$) is the number of positive (negative) colliding particles. Property \ref{proB} is not strictly necessary for continuing the evolution of $X_\tau$ after $\tau$. However, if Property \ref{proB} holds, then the drift is smooth in an open region around $X_\tau$, for which it is simple to continue the solution in the strong sense.

\subsection{Extremal cases of \eqref{SDE}}
\label{s:intro:extre}
To get a sense of the kind of collisions one may expect, we first consider the asymptotic ends of the parameter range of the interaction strength $\gamma \in (0,\infty)$. We do not claim that the corresponding systems arise from \eqref{SDE} by taking limits of $\gamma$. We interpret the system corresponding to $\gamma = 0$ as $(X^1_t)_{t\ge 0},\dots,(X^N_t)_{t\ge 0}$ being independent Brownian motions in $\R^2$. In particular, there are no collisions. The system corresponding to $\gamma = \infty$ is formally obtained by rescaling \eqref{SDE} in time to obtain a drift independent of $\gamma$, and then formally taking $\gamma \to \infty$. This results in the system of deterministic ODEs given by
\begin{equation} \label{ODE:x}
  \frac d{dt} x^i = \sum_{j =1 }^N b^i b^j f(x^i - x^j) \qquad \text{for } t \in (0, \tau), \ i = 1,\ldots,N,
\end{equation}
where $\tau \geq 0$ is the first collision time.
This system is of interest on its own right, e.g.\ as a model for vortices or dislocations.
While it is known that Property \ref{proA} holds for \eqref{ODE:x} (see e.g.\ \cite{BethuelOrlandiSmets07}), Property \ref{proB} fails for certain initial conditions (to appear in \cite{VanMeursMurokawaOshikawaXX}), which shows that there is no natural unique manner to continue the solution after collision. 
To demonstrate the complexity of collisions of \eqref{ODE:x}, we give two examples. First, for the circular configuration in Figure \ref{fig:spec:X}(a), all particles will collide in finite time at the center point. The choice of the radius is irrelevant, because solutions to \eqref{ODE:x} satisfy the following scaling invariance: if $x(t)$ is a solution, then $y(t) := \frac1L x(L^2 t)$ is too for any $L > 0$.
Also, the number of particles $N$ is irrelevant; $N$ only needs to be even. Hence, this example shows that there is no limit to the number of particles that can collide for \eqref{ODE:x}. Second, the configuration in Figure \ref{fig:spec:X}(b) shows a stationary point of \eqref{ODE:x}. By the scaling invariance it follows that particles can get arbitrarily close without colliding. 

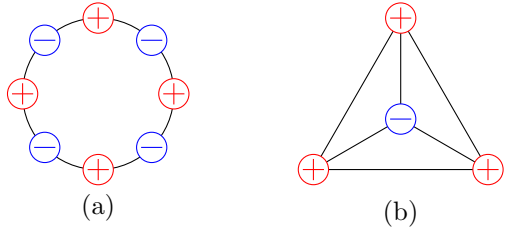
\begin{figure}[h]
\centering
\begin{tikzpicture}[scale=1]
    \def \r {.2}
    
    \draw (0,0) circle (1);
    
    \foreach \i in {0,...,3}{
    \begin{scope}[rotate = \i*90]
      \begin{scope}[shift={(1,0)}] 
      \filldraw[draw=red, fill=white] (0, 0) circle (\r); 
      \draw[red] (-.7*\r, 0) -- (.7*\r, 0);
      \draw[red] (0, -.7*\r) -- (0, .7*\r);
      \end{scope}
    \end{scope}
    }
    
    \foreach \i in {0,...,3}{
    \begin{scope}[rotate = 45 + \i*90]
      \begin{scope}[shift={(1,0)}, rotate = 45 - \i*90] 
        \filldraw[draw=blue, fill=white] (0, 0) circle (\r); 
        \draw[blue] (0,-.7*\r) -- (0,.7*\r);
      \end{scope}
    \end{scope}   
    }   
    
    \draw (0,-1-\r) node[below]{(a)};        
    
    \begin{scope}[shift={(4,-.333)}, scale = 1.333]
    \draw (0,1) -- (0.866,-.5) -- (-0.866,-.5) -- cycle;
    \foreach \i in {0,1,2}{
    \begin{scope}[rotate = 90 + \i*120]
      \draw (0,0) -- (1,0);
      \begin{scope}[shift={(1,0)}, rotate = - 90 - \i*120, scale = .75] 
      \filldraw[draw=red, fill=white] (0, 0) circle (\r); 
      \draw[red] (-.7*\r, 0) -- (.7*\r, 0);
      \draw[red] (0, -.7*\r) -- (0, .7*\r);
      \end{scope}
    \end{scope}
    }
    \begin{scope}[scale = .75]
    \filldraw[draw=blue, fill=white] (0, 0) circle (\r); 
    \draw[blue] (-.7*\r, 0) -- (.7*\r, 0);
    \draw (0,-.75-\r) node[below]{(b)}; 
    \end{scope}
    \end{scope}
\end{tikzpicture}
\caption{(a) A particle configuration with all particles equispaced on a circle with alternating signs. (b) 3 positive particles placed on an equilateral triangle with a negative particle in the center.}
\label{fig:spec:X}
\end{figure}

\subsection{Main results} 
\label{s:intro:main:results}
Our main result, Theorem \ref{t:SDE}, shows that the behavior of the SDE system \eqref{SDE} is somewhere in between the two extremal cases described in Section \ref{s:intro:extre}. In more detail,
for each $\gamma > 0$, each $N \geq 2$, each $b$ and each initial condition $X_0$ of distinct particles,  
\begin{enumerate}[label=(\alph*)]
  \item \label{thma} there exists a strong, global solution,
  \item \label{thmb} the solution is path-wise unique, 
  \item \label{thmc} each collision happens at a distinct time,
  \item \label{thmd} each collision is \textit{simple}, i.e.\ between two particles of opposite sign,
  \item \label{thme} collisions keep on happening in finite time until only particles of the same sign are left.
\end{enumerate}

Let us compare Properties \ref{thma}--\ref{thme} with the extremal cases from Section \ref{s:intro:extre}. 
With respect to the extremal case $\gamma = 0$ of independent Brownian motions, \ref{thma} and \ref{thmb} are shared, \ref{thmc} and \ref{thmd} become empty statements (because collisions do not occur), but \ref{thme} is completely different. With respect to the extremal case $\gamma = \infty$ given by the ODE system \eqref{ODE:x}, \ref{thma} and \ref{thmb} are shared
until the first multi-particle collision at which two or more particles remain at the collision site after applying the collision rule. Moreover, each of \ref{thmc}, \ref{thmd} and \ref{thme} are different due to the complicated collisions that may (or will not) occur in \eqref{ODE:x}. In particular, with \ref{thmc}--\ref{thme} on \eqref{SDE} established in this paper, we obtain a precise meaning to the \textit{instability} of stationary solutions and multi-particle collisions of \eqref{ODE:x}. Indeed, when \eqref{ODE:x} is perturbed with noise, our results show that multi-particle collisions have to break down into several simple collisions, and that stationary solutions (which only exist when both signs appear) break down and have a collision in finite time, irrespective of the size of the noise.
As a side result, we also show (Proposition \ref{p:ODE}) that the solution $x(t)$ of the ODE system \eqref{ODE:x} is left-continuous at its first collision. This was already proved in \cite{SmetsBethuelOrlandi07}, but we use a completely different, more direct, ODE-based 
proof\footnote{The proof in \cite{SmetsBethuelOrlandi07} starts from the complex parabolic Ginzburg-Landau equation with small parameter $\e$. It is shown that \eqref{ODE:x} appears in the limit $\e \to 0$.}. This proof serves as a blueprint for the proof of the left-continuity of the process $X_t$ at collision. The proof is inspired by \cite[Section 5]{Tardy24}, which we describe below in Sections \ref{s:intro:lit} and \ref{s:ODE:ex}.

\subsection{Building blocks of the proof} 
\label{s:intro:proof}

We prove Theorem \ref{t:SDE} by iterating over the collision times. Then, the main challenge is to establish Properties \ref{proA} and \ref{proB} from Section \ref{s:intro:aim} at the first collision time, which we do in Proposition \ref{p:cty:tau} and Theorem \ref{t:3plus}. This requires control on the singularity of the drift. Yet, there is no monotonicity or convexity available for this.

To get control on the drift, we use five tools from the literature that go back until at least \cite{SmetsBethuelOrlandi07}. 
We briefly recall them. For simplicity, we first consider the ODE system \eqref{ODE:x}. Similar to \eqref{SDE:vec-form} we write it in vector form as 
\begin{equation} \label{ODE:vec-form}
  \frac d{dt} x = F(x), \qquad F^i(x) = \gamma \sum_{j =1 }^N b^i b^j f(x^i - x^j),
\end{equation}
where we consider $\gamma > 0$ for later use, and consider the solution $x$ until the first collision time $\tau$.

The first tool is the following observation on the drift $F$:
\begin{equation*}
  \sum_{i=1}^N F^i(y) = 0 \qquad \text{for all } y \in (\R^2)^N.
\end{equation*}
This is a simple consequence from the summand in \eqref{ODE:vec-form} being odd in $i,j$. As a consequence, we get that the empirical mean (or center of mass) $M(x)$ defined by 
\begin{equation*}
  M : (\R^2)^N \to \R^2,  \qquad
    M(y) := \frac1N \sum_{i =1 }^N y^i
\end{equation*}
is stationary, i.e.\ $\frac d{dt} M(x) = 0$.

The second tool is also a property of $F$. Recalling $f(0) = 0$ and noting that $z \cdot f(z) = 0$ for all $z \in \R^2 \setminus \{0\}$, a simple computation show that for all  $y \in (\R^2)^N$
\begin{align} \notag
   y \cdot F(y) 
   &= \gamma \sum_{i,j=1}^N  b^i b^j y^i \cdot f(y^i - y^j)
   = \frac\gamma2 \sum_{i,j=1}^N  b^i b^j (y^i - y^j) \cdot f(y^i - y^j) \\\label{yFy}
   &= \frac\gamma2 \Big( \sum_{i,j=1}^N  b^i b^j - \sum_{i=1}^N (b^i)^2 \Big)
   = \frac12 \gamma \Big( \Big( \sum_{i=1}^N b^i \Big)^2 - N \Big),
\end{align}
which is constant in $y$. As a consequence, we obtain
\begin{equation*}
  \frac d{dt} R(x) = \gamma\Big( \Big( \sum_{i=1}^N b^i \Big)^2 - N \Big),
\end{equation*}
where
\begin{equation*} 
  R : (\R^2)^N \to \R,  \qquad
  R(y) 
  := \sum_{i=1}^N |y^i - M(y)|^2 
  = \frac1{2 N} \sum_{i=1}^N \sum_{j=1}^N |y^i - y^j|^2
\end{equation*} 
is, up to a multiplicative factor, the dispersion (or the second moment centered at $M(y)$).
Hence, $R(x(t))$ is affine in time. In particular, it does not blow up, no matter how close $x(t)$ is to the singularity of $F$. Also, if $R(x(t))$ is nondecreasing, then there cannot be an all-particle collision. Thus, if $( \sum_{i=1}^N b^i )^2 \geq N$, then the all-particle collision does not happen.
The reason for centering $R(y)$ at $M(y)$, is that $R(x(t))$ reaching $0$ is equivalent to the all-particle collision.

The third tool is to realize that $M$ and $R$ can also be constructed for any particle \textit{cluster} 
\begin{equation} \label{xK}
    x^K := \{x^i : i \in K\} \in (\R^2)^{|K|},
\end{equation}
where $K \subset \bbrack{1,N}$ is nonempty and $|K|$ is the number of elements in $K$. Imagine a configuration $x$ for which there exists a cluster $x^K$ such that for each $i,j \in K$ and each $k \notin K$, $x^i$ and $x^j$ are close to each other while $x^i$ and $x^k$ are sufficiently separated. Set
\begin{equation} \label{RK}
  M^K(x) := \frac1{|K|} \sum_{i \in K } x^i, \qquad
  R^K(x) := \sum_{i \in K } |x^i - M^K(x)|^2 
  = \frac1{2 |K|} \sum_{i \in K} \sum_{j \in K} |x^i - x^j|^2
\end{equation} 
as the local empirical mean and local dispersion.
Because all particles in the cluster are close to each other, $R^K(x)$ is small and each particle $x^i$ is close to $M(x)$. Because all particles in the cluster are separated from all other particles, we obtain from similar computations done above that $\frac d{dt} M^K(x)$ and $\frac d{dt} R^K(x)$ remain bounded. This gives control on both the position and the diameter of the cluster, irrespective of how close $x$ gets to the singularity of $F$. This is the key tool for handling the singularity of $F$. Moreover, as a consequence of the second tool, $R^K(x(t))$ reaching $0$ is equivalent to the occurrence of the multi-particle collision of all particles in the cluster. Thus, we can characterize any multi-particle collision by $R^K(x(t))$ reaching $0$ for the corresponding set $K$ of colliding particles. 

The fourth tool is the extension of the previous three tools to the SDE system \eqref{SDE}. A simple, formal computation based on It\^o's lemma shows that $M(X_t)$ is a Brownian motion and that $R(X_t)$ is a squared Bessel process of dimension 
  \begin{equation} \label{delta}
  \delta := \gamma \bigg( \Big( \sum_{i = 1}^N b^i \Big)^2 - N \bigg) + 2(N-1).
\end{equation}
We recall Bessel processes and their properties in Section \ref{s:BP}. In particular, if $\delta < 2$, then $R(X_t)$ a.s.\ hits $0$ in finite time, but otherwise ($\delta \geq 2$), $R(X_t)$ never hits $0$. This sharp phase transition at $\delta = 2$ is the key for proving whether the all-particle collision can (and then will) or cannot occur. 
As in the third tool, we will consider the general, localized versions $M^K(X_t)$ and $R^K(X_t)$. These processes are more complicated, but with a careful application of Girsanov's Theorem we can still show that $M^K(X_t)$ behaves sufficiently as a Brownian motion and that $R^K(X_t)$ behaves sufficiently as a squared Bessel process of dimension 
  \begin{equation} \label{deltaK}
  \delta_K := \gamma \bigg( \Big( \sum_{i \in K} b^i \Big)^2 - |K| \bigg) + 2(|K|-1).
\end{equation}

The fifth tool is the space-time scaling invariance of $X_t$ (see Proposition \ref{p:invar}\ref{p:invar:scale}), which roughly says that $Y_s := \frac1L X_{L^2 s}$ is also a solution for any $L > 0$. This tool allows us to study clusters $X_t^K$ of close particles on a unit spatial scale, from which we can construct events on the behavior of $X_t^K$ that happen with positive probability uniformly over the closeness of the particles.

Our proofs of Properties \ref{thma}--\ref{thme} make repeated and iterative use of these five tools to untangle the most complex particle configurations and show that the system \eqref{SDE} behaves in the orderly manner dictated by \ref{thma}--\ref{thme}.  

\subsection{Contribution to the literature} 
\label{s:intro:lit}

We start by mentioning three special cases. First, in the single-sign case, i.e.\ when $b^i = 1$ for all $1 \leq i \leq N$, \cite{LiuYang16} has proved \ref{thma} and \ref{thmb} from Section \ref{s:intro:main:results}, and showed that collisions do not occur, which is consistent with \ref{thmc}--\ref{thme}. Second, when $N=2$, \ref{thma}--\ref{thme} are established in \cite[Section 6]{FournierJourdain17}. Third, in one spatial dimension, \cite{VanMeursPeletierPozar22} studies the ODE system \eqref{ODE:x} globally in time with the removal rule. For this system, they show that \ref{thma} and \ref{thmb} hold, and that \ref{thmc} and \ref{thmd} do not hold in general. It remains unclear whether \ref{thme} holds.
\medskip

Next we review recent work on the closely related Keller-Segel system, which is obtained from \eqref{SDE} by replacing `$b^i b^j$' by `$-1$', i.e.
\begin{equation} \label{KS}
  (KS) \qquad dX_t^i = dB_t^i - \gamma \sum_{j =1 }^N f(X_t^i - X_t^j) dt,
\end{equation}
and applying instant separation upon collision.
In $(KS)$, all particles are indistinguishable (i.e.\ no sign is allocated to the particles) and attract each other. The well-posedness of $(KS)$ is studied in 
\cite{CattiauxPedeches16,
FournierJourdain17,
FournierTardy25,
Tardy24}.

The instant separation rule raises two additional challenges with respect to $(\SC)$ from \eqref{SDE}:
\begin{itemize}
  \item the set of collision times is infinite,
  \item collisions of more than 2 particles can occur. 
\end{itemize}
The infinite set of collision times requires a weaker notion of solutions. In \cite{FournierTardy25} this notion is based on Dirichlet form theory.
Moreover, due to the pairwise attraction between all particles, the number of colliding particles may exceeds a critical number (which depends on $\gamma$), in which case the colliding particles cannot separate anymore. When this happens, the particle system is said to explode, and the solution ceases to exist. 

Thanks to the similarity of $(\SC)$ with $(KS)$, we will prove Proposition \ref{p:cty:tau} by modifying the corresponding proofs in \cite[Section 5]{Tardy24} or \cite[Theorem 5-(i)]{FournierTardy25} on $(KS)$. We expect that Theorem \ref{t:3plus} can also be proved by modifying the arguments in \cite[Section 9]{FournierTardy25}. Along this route, $x$ is decomposed as the triple $(M(x), R(x), U(x))$, where the `angle' $U(x) \in \cS \subset (\R^2)^N$ is given by $U^i(x) := (x^i - M(x))/\sqrt{R(x)}$ for $i = 1,\ldots,N$, and $\cS$ is the intersection of the unit sphere in $(\R^2)^N$ and the hyperplane of co-dimension $2$ given by $\{ y \in (\R^2)^N : M(y) = 0 \}$. For $(KS)$, the processes $M(X_t)$, $R(X_t)$ and $U(X_{\rho_t})$ are independent for a suitable time change $\rho_t$. A similar computation that demonstrates this independence also applies to $(\SC)$, and hence the same decomposition into independent processes holds for $(\SC)$.

However, this route turned out to be less convenient for $(\SC)$ than for $(KS)$: most of the difficulties of $X_t$ get transferred to $U(X_t)$ while the SDE for $U(X_t)$ is more complicated. 
The reason it works well for
$(KS)$ is that $(KS)$ has a certain monotonicity property that $(\SC)$ does not have. Indeed, while for both systems $R^K(X_t)$ behaves as a squared Bessel process for $K$ as described below \eqref{xK}, the dimension $\delta_K$ of $(KS)$ is decreasing in $|K|$. This allows for an iterative argument over the size of $K$ to get control on multi-particle collisions. For $(\SC)$ this monotonicity of $\delta_K$ does not hold because of the signs $b$ (\eqref{deltaK} clarifies the dependence). 

In our proof of Theorem \ref{t:3plus}, we do not apply the splitting above as done for $(KS)$. Instead, we rely on the fifth tool given by the space-time scaling invariance.
This results in quite a different proof construction than that in \cite[Section 9]{FournierTardy25}; see Section \ref{s:it}. We consider this the main technical novelty of our paper.

Actually, the scaling invariance of $(\SC)$ also holds for $(KS)$. In fact, the decomposition of $X_t$ into  $M(X_t)$, $R(X_t)$ and $U(X_{\rho_t})$ involves a similar space-time scaling.
We expect that our proof method can be adapted to $(KS)$, which potentially results in a simpler proof of the collision properties for $(KS)$.
\medskip

Next we turn to \cite{VanMeursPeletierSlangen25}. There, $(\SC)$ \textit{with the instant separation rule} is studied with the techniques developed in \cite{FournierJourdain17} on $(KS)$. The first result \cite[Theorem 3.1]{VanMeursPeletierSlangen25} states that if a weak solution exists and if $b$ contains both signs, then collisions happen with positive probability on any time interval. We upgrade that proof here to a more streamlined version, which moreover gives a quantitative result; see Proposition \ref{p:2coll} below. It becomes part of our proof of Theorem \ref{t:3plus}, and results in the stronger Property \ref{thmd} that each collision has to be simple.

The second result \cite[Theorem 4.1]{VanMeursPeletierSlangen25} is the construction of a weak solution, which works only for $\gamma < \frac12$. In addition, it is shown that this solution satisfies \ref{thmd}. There is no statement on the uniqueness of solutions; the only result available on this is \cite[Section 6]{FournierJourdain17} for the special case $N=2$.

In comparison to the results in this paper, we observe that $(\SC)$ with the removal rule has several advantages over the instant separation rule. Indeed, the removal rule fits better to the applications mentioned in Section \ref{s:intro:SC} and results into stronger properties (i.e.\ \ref{thma}--\ref{thme}).

\subsection{Discussion}

Here we address possible extensions of our results. First, the extension from the Coulomb interaction force $f(x) = x/|x|^2$ to the Riesz interaction force $f(x) = x/|x|^\alpha$ for $\alpha > 1$ would be very interesting given the large literature on the Riesz interaction force; see e.g.\ the review \cite{Lewin22}. The main challenge is that the second tool breaks down, i.e.\ \eqref{yFy} will depend on $y$, and thus the equation for the dispersion $R(X_t)$ will not decouple. 

Second, \eqref{yFy} does hold in higher dimensions for $f(x) = x/|x|^2$ with $x \in \R^d$ and $d \geq 3$. Then, $f$ is sub-Coulomb. The five tools from Section \ref{s:intro:proof} will still hold with the modification that the `$2$' in the formula's for $\delta$ and $\delta_K$ is replaced by $d$ (see the proof of Lemma \ref{l:R:sBP} for details). This suggests that our proofs may be modifiable to the case $d \geq 3$. We do not pursue this here and leave it for future work.

Third, we address the question on a possible mean-field limit of \eqref{SDE}, i.e.\ the limit $N \to \infty$ after a proper rescaling is applied. For the collision rule given by instant separation, such mean-field limits are established in \cite{FournierJourdain17,VanMeursPeletierSlangen25} for $\gamma = \frac1N \tilde \gamma$ with $\tilde \gamma$ small enough. For $(\SC)$ with instant separation, the mean-field limit is a coupled set of two equations for the particle densities $\rho^\pm$ of the positive and the negative particles. Both equations are conservative and both contain  a diffusion term and a nonlocal interaction term. For $(\SC)$ with the removal rule, we do not expect the mean-field limit equations to be conservative. The expected equations are studied for regular initial conditions in \cite[Theorem 3.1]{MasmoudiZhang05} and by a variational approach in \cite{AmbrosioMaininiSerfaty11}. However, neither of these results provide well-posedness for weak solutions, which poses a second obstacle for establishing the mean-field limit of $(\SC)$.

\medskip
The paper is organized as follows. Section \ref{s:ODE:ex} establishes the continuity of solutions to the ODE system \eqref{ODE:x} at collision. In Section \ref{s:BP} we recall several basic properties of Bessel processes. In Section \ref{s:SC} we develop the tools mentioned in Section \ref{s:intro:proof} for handling $(\SC)$. In Section \ref{s:SDE:cty} we prove Proposition \ref{p:cty:tau} on the continuity of the solution at the first collision time. In Section \ref{s:simple-coll} we prove Theorem \ref{t:3plus} on the characterization of the first collision. In Section \ref{s:mainT} we put these results together to prove our main Theorem \ref{t:SDE} on the global solution of \eqref{SDE}.

\section{Continuity of the ODE system at collisions}
\label{s:ODE:ex}

Consider the ODE system \eqref{ODE:vec-form}, i.e.\ the deterministic counterpart of the SDE \eqref{SDE}. By rescaling time we may set $\gamma =1$. The main result in this section is Proposition \ref{p:ODE} below on the continuity of the solution at collision. We consider this both as an interesting result on its own and as a blueprint for the stochastic version of this result (i.e.\ Proposition \ref{p:cty:tau}). The proof is inspired by \cite[Section 5]{Tardy24} and \cite[Theorem 5-(i)]{FournierTardy25}. 

We set
\begin{equation} \label{cE}
  \cE := \{ x \in (\R^2)^N : x^i \neq x^j \text{ for all } 1 \leq i < j \leq N \}
\end{equation}
as the open set where $F$ is regular (smooth actually). Note that $\cE$ is the set of all particle configurations with distinct particle positions. 

\begin{prop}[Continuity of the ODE solution at collision] \label{p:ODE}
    Given $N \geq 1$, $b \in \{-1,1\}^N$ and $x_\circ \in \cE$, let $x$ be the solution of \eqref{ODE:vec-form} with $x(0) = x_\circ$ and $\tau \in (0, \infty]$ the maximal time of existence. If $\tau < \infty$, then the limit $x(\tau-) := \lim_{t \to \tau} x(t)$ exists and $x(\tau-) \notin \cE$.
\end{prop}
 
The difficulty in proving Proposition \ref{p:ODE} is that $F$ is singular at $\partial \cE$, which contains all possible kinds of (multi-)particle collisions, and that $F$ is unbounded near $\partial \cE$. Hence, a priori it is unclear whether any of the limits $x^i(\tau-)$ exists; the positions $x^i$ may oscillate or become unbounded. To make matters worse, it is known (recall e.g.\ Figure \ref{fig:spec:X}) that collisions between any number of particles can occur.

Based on these difficulties, our proof strategy is to analyze the liminf and limsup as $t \to \tau$ of $R^K(x(t)) \geq 0$ for \textit{each} nonempty cluster index set $K \subset \bbrack{1,N}$, where we recall $R^K(x)$ from \eqref{RK} to be the local dispersion. In particular, we use $R^K(x)$ as a measure for how close together the particles in $K$ are, where $R^K(x) = 0$ corresponds to all particles in $K$ being at the same position, which is then given by the local empirical mean $M^K(x)$. As can be guessed from the computations in Section \ref{s:intro:proof}, $R^K(x(t))$ and $M^K(x(t))$ are regular in time if the particles inside $K$ remain separated from those outside of $K$, even if the particles within $K$ get arbitrarily close. Motivated by this, we will construct a partition $\bK$ of $\bbrack{1,N}$ (we interpret $\bK \subset \{K : K \subset \bbrack{1,N}\}$) such that $\liminf_{t \to \tau } R^K(x(t)) = 0$ for each $K \in \bK$, while particles belonging to different $K, K' \in \bK$ remain separated. This will give us enough control to show that for each $K \in \bK$ the limit $R^K(x(\tau-))$ exists (and equals $0$). By the separation property, this implies that the points $(M^K(x))_{K \in \bK}$ are separated. This will give sufficient control to prove Proposition \ref{p:ODE}. 

\begin{proof}[Proof of Proposition \ref{p:ODE}]
Note from $\tau$ being the maximal time of existence and $\tau < \infty$ that either 
\begin{equation} \label{pfyj}
\limsup_{t \to \tau} |x(t)| = \infty
\quad \text{or} \quad
\liminf_{t \to \tau} \dist(x(t), \cE) = 0.
\end{equation}

To build $\bK$ as above, we start with preparations. We recall $M(x), R(x), M^K(x), R^K(x)$ and their properties from Section \ref{s:intro:proof}. In particular,
\begin{equation} \label{pfxd}  
\frac d{dt} M(x) = 0, \qquad \frac d{dt} R(x) = C
  \qquad \text{on } (0,\tau)
\end{equation} 
for some constant $C \in \R$. 
Hence, the limits $M(x(\tau-))$ and $R(x(\tau-))$ exist. Thus, the first of the two cases in \eqref{pfyj} cannot hold. Then, from the remaining second case we obtain
\begin{equation} \label{pfyi}
  \liminf_{t \to \tau} |x^i - x^j|(t) = 0
  \qquad \text{for some } i \neq j.
\end{equation} 
In addition, for any $K \subset \bbrack{1,N}$ with $|K| \leq N-1$, we have
\begin{align} \label{pfzv}
  \frac d{dt}  M^K(x)
  = \frac1{|K|} \sum_{i \in K} \sum_{j \notin K} b^i b^j f(x^i - x^j)
  \qquad \text{on } (0,\tau). 
\end{align} 
The main ingredient of the proof is stated in Lemma \ref{l:ODE:ex} below. It states that for any nonempty $K \subset \bbrack{1,N}$
\begin{equation} \label{pfxw}
  \liminf_{t \nearrow \tau} R^K(x (t)) > 0 
  \quad \text{or} \quad 
  \lim_{t \nearrow \tau} R^K(x (t)) = 0.
\end{equation}

Next we build $\bK$ and complete the proof. On $\bbrack{1,N}$ let the relation $i \sim j$ be defined by $\liminf_{t \nearrow \tau} R^{\{i,j\}}(x(t)) = 0$. Clearly, this is an equivalence relation. Take $\bK$ as the family of equivalence classes. Note from \eqref{pfxw} that for all $K \in \bK$ and all $j \notin K$ 
\begin{equation*} 
  R^K(x(\tau-)) = 0
  \quad \text{and} \quad
  \liminf_{t \nearrow \tau} R^{K \cup \{j\}}(x(t)) > 0.
\end{equation*}
Hence, for each distinct $K,K' \in \bK$, each $i \in K$ and each $j \in K'$
\begin{equation} \label{pfxv}
  \lim_{t \nearrow \tau}  |x^i(t) - M^K(x(t))| = 0
  \quad \text{and} \quad\liminf_{t \nearrow \tau} |x^i(t) - x^j(t)| > 0.
\end{equation}
Fix $K \in \bK$. By \eqref{pfzv} and \eqref{pfxv}, $\frac d{dt} M^K(x(t))$ is bounded on $(0, \tau)$. Thus, $M^K(x(\tau-))$ exists for each $K \in \bK$. Then, by \eqref{pfxv}, $x^i(\tau-)$ exists for all $i \in K$. Since $K \in \bK$ is arbitrary, we obtain that $x(\tau-)$ exists. Then, \eqref{pfyi} guarantees that $x(\tau-) \notin \cE$. 
\end{proof}

\begin{lem} \label{l:ODE:ex} 
Let $N,b,x_\circ, x, \tau$ be as in Proposition \ref{p:ODE} and $K \subset \bbrack{1,N}$ be nonempty. Then
\begin{equation*}
  \liminf_{t \nearrow \tau} R^K (x(t)) > 0 
  \quad \text{or} \quad 
  \lim_{t \nearrow \tau} R^K (x(t)) = 0.
\end{equation*}
\end{lem}

\begin{proof} 
In addition to the preparations done in the proof of Proposition \ref{p:ODE}, we make the following observations.
First, when $K = \{i\}$ is a singleton, we simply have $ M^{\{i\}}(x) = x^i$ and $ R^{\{i\}}(x) = 0$.
Second, for any $K \subset \bbrack{1,N}$ with $2 \leq |K| \leq N-1$, we have
\begin{align} \label{pfzr}
  \frac d{dt}  R^K(x) 
  = \Big( \sum_{i \in K} b^i \Big)^2 - |K| + 2 \sum_{i \in K} \sum_{j \notin K} b^i b^j (x^i -  M^K(x)) \cdot f(x^i - x^j).
\end{align} 
Third, we note from the second expression of $R^K$ in \eqref{RK} that it is ordered in the following sense: for all $y \in (\R^2)^N$
\begin{equation} \label{pfzt}
  K \subset K' \implies  R^K(y) \leq \frac{|K'|}{|K|} R^{K'}(y).
\end{equation}
In particular, if $R^{K'}(x) = 0$, then $R^K(x) = 0$ for all $K \subset K'$.

To prove Lemma \ref{l:ODE:ex} it is enough to show that the following induction hypothesis holds for all $n \in \bbrack{1,N}$:
\begin{equation} \tag{IH} \label{pfzu} 
  \liminf_{t \nearrow \tau} R^K(x(t))  = 0
  \implies
  \limsup_{t \nearrow \tau} R^K(x(t)) = 0
  \qquad
  \text{for all } K \subset \bbrack{1,N} \text{ with } |K| = n.
\end{equation}
From the properties of $R^{\{i\}}(x)$ and $R(x)$ established in the preparations we see that \eqref{pfzu} holds for $n \in \{1,N\}$. We prove \eqref{pfzu} for the remaining values of $n$ by backward induction. Suppose \eqref{pfzu} holds for some $n \in \bbrack{3,N}$. Take any $K \subset \bbrack{1,N}$ with $|K| = n-1$. We will show that the implication in \eqref{pfzu} holds for $K$. 

Set $K^j := K \cup \{j\}$ for any $j \notin K$.
Suppose that there exists $j \notin K$ such that $\liminf_{t \nearrow \tau} R^{K^j}(x(t))$ $= 0$. Then, by \eqref{pfzt} and \eqref{pfzu} applied with $n$ we obtain
\begin{equation*}
  \limsup_{t \nearrow \tau}  R^{K}(x(t))
  \leq \frac n{n-1} \limsup_{t \nearrow \tau} R^{K^j}(x(t))
  = 0,
\end{equation*}
and thus $K$ satisfies \eqref{pfzu}. 

Hence, we may assume that
\begin{equation*}
  \min_{j \notin K} \liminf_{t \nearrow \tau}  R^{K^j}(x(t)) > 0.
\end{equation*}
Then, there exist constants $\alpha, u > 0$ such that
\begin{equation} \label{pfzs}
   R^{K^j}(x(t)) \geq 7 \alpha
  \qquad
  \text{for all } t \in [\tau - u, \tau) \text{ and all } j \notin K.
\end{equation} 
For a given $j \notin K$, let $\iota \in K$ be a minimizer of $d^j(t) := \min_{i \in K} |x^i(t) - x^j(t)|$. For any $i \in K \setminus \{\iota\}$
\begin{equation*}
  |x^j(t) - x^i(t)|
  \leq |x^j(t) - x^\iota(t)| + |x^\iota(t) - M^K(x(t))| + |M^K(x(t)) - x^i(t)|
  \leq d^j(t) + 2 \sqrt{R^K(x(t))}.
\end{equation*}
Thus, for any $i \in K$ and any $j \notin K$, we have $|x^j(t) - x^i(t)|^2 \leq 2(d^j(t) )^2 + 8 R^K(x(t))$. Using this, we obtain on $[\tau - u, \tau)$ that, by \eqref{pfzs} and $n \geq 3$, for any $j \notin K$
\begin{equation} \label{pfzp}
  7 \alpha
  \leq R^{K^j}(x(t))
  = \frac n{n-1} R^K(x(t)) + \frac1{2n} \sum_{i \in K} |x^i(t) - x^j(t)|^2
  \leq 6 R^K(x(t)) + (d^j(t))^2.
\end{equation}

Next we conclude by showing that for any sequence $s_\ell \nearrow \tau$ as $\ell \to \infty$, we have $R^K (x(s_\ell)) \to 0$ as $\ell \to \infty$. We apply a bootstrap argument. On time intervals $[a,b) \subset [\tau - u, \tau)$ on which $R^K(x) < \alpha$, we have by \eqref{pfzp} that $d^j > \sqrt \alpha$ for any $j \notin K$. This yields (recall \eqref{pfzr})
\begin{align} \notag
  \Big| \frac d{dt} R^K(x) \Big| 
  &= \Big| \Big( \sum_{i \in K} b^i \Big)^2 - n + 2 \sum_{i \in K} \sum_{j \notin K} b^i b^j (x^i - z^K) \cdot f(x^i - x^j) \Big|  \\\label{pfxu}
  &\leq \Big|\Big( \sum_{i \in K} b^i \Big)^2 - n \Big| + 2 n \sum_{j \notin K} \frac{\sqrt{R^K(x)}}{d^j} 
  < C + 2 n(N-n) \frac{\sqrt{\alpha}}{\sqrt{\alpha}}
  =: C', 
\end{align} 
which is uniformly bounded in $a,b$. To use this, we employ $\liminf_{t \nearrow \tau } R^K(x(t)) = 0$ to obtain a 
sequence $t_k \nearrow \tau$ with $R^K (x(t_k)) \to 0$ as $k \to \infty$. Let $k_* \geq 1$ be such that for all $k \geq k_*$ we have $R^K (x(t_k)) \leq  \frac\alpha2$ and $t_{k+1} - t_k \leq \alpha / (2C')$. Then, for any $k \geq k_*$, applying \eqref{pfxu} on $[a,b)$ with $a = t_k$ and some $b$ close enough to $a$ to ensure that $R^K(x) < \alpha$ on $[a,b)$, we obtain that 
\[
R^K(x(t)) \leq R^K(x(a)) + C'(t-a) \leq \frac\alpha2 + C'(t - t_k)
\qquad \text{on } [t_k,b).
\]
In particular, for any $t < t_k + \alpha / (2C')$, this estimate yields $R^K(x(t)) < \alpha$, and thus we may take $b = t_{k+1}$. 
Finally, recalling $s_\ell$, we have for all $\ell$ large enough that $s_\ell \in [t_{k_*}, \tau)$. For any such $\ell$, let $k_\ell$ be such that $s_\ell \in [t_{k_\ell}, t_{k_\ell + 1})$. Clearly, $k_\ell \to \infty$ as $\ell \to \infty$. Then, by \eqref{pfxu} on $[t_{k_\ell}, t_{k_\ell + 1})$, we obtain
\[
  R^K (x(s_\ell)) 
  \leq R^K (x(t_{k_\ell})) + C' (s_\ell - t_{k_\ell})
  \leq R^K (x(t_{k_\ell})) + C' (t_{k_\ell + 1} - t_{k_\ell}),
\]
which vanishes as $\ell \to \infty$.
\end{proof}

\section{Squared Bessel processes stopped at $0$}
\label{s:BP}

In this section we recall several established properties on squared Bessel processes and introduce related notation.
Let $\delta \in \R$ be a parameter, $r \geq 0$ be the initial condition, and $(\beta_t)_{t \ge 0}$ be a $1$-dimensional Brownian Motion. We consider the following SDE
\begin{equation} \label{R:Bessel:SDE}
  R_t = r + \delta t + \int_0^t 2 \sqrt{ R_s } d \beta_s.
\end{equation}
The typical example of a process that satisfies \eqref{R:Bessel:SDE} for all $t \geq 0$ for some $(\beta_t)_{t \ge 0}$ with $\delta = n$ a positive integer, is $(| B_t |^2)_{t \ge 0}$, where $(B_t)_{t \ge 0}$ is an $n$-dimensional Brownian motion. Based on this, $\delta$ is interpreted as the dimension, even when $\delta \in \R \setminus \N$.

\begin{lem}[Well-posedness of \eqref{R:Bessel:SDE} for $\delta \in \R$; {\cite[\S XI.1]{RevuzYor99}}]
  \label{l:sqBessel}
\comm{[ RY05 refers for this to Chap IX Sect 3]}  
  \noindent
\begin{enumerate}
\item   \label{i:l:sqBessel:pos-dim}
For $\delta > 0$ there exists a unique strong solution $(R_t)_{t \ge 0}$ of \eqref{R:Bessel:SDE}.
\item  For $\delta\le0$ there exists a unique strong solution $(R_t)_{t \in [0, \sigma)}$ of \eqref{R:Bessel:SDE}, where $\sigma$ is the hitting time at $0$. 
\label{i:l:sqBessel:neg-dim}
\end{enumerate}
\end{lem}

For our purpose, it is convenient, for $\delta \leq 0$, to extend $(R_t)_{t \in [0, \sigma)}$ in time by freezing it at $0$, i.e.\ by defining $R_t = 0$ for all $t \ge \sigma$. To make this consistent with the case $\delta > 0$, we also freeze $R_t$ at $0$ upon its first $0$-hitting time $\sigma$. This defines, for all $\delta \in \R$, the squared Bessel process $(R_t)_{t \ge 0}$ stopped at $0$; we will refer to such a process as a $\SB(\delta, r)$-process, omit the second entry if no initial condition is specified, and omit the wording `stopped at $0$' (we will rarely work with squared Bessel process not stopped at $0$, and indicate such cases clearly).

\begin{thm}[Properties] \label{t:BP}
Let $(\Omega, \cF,(\cF_t)_{t\ge 0},\PP)$ be a filtered probability space, $\delta \in \R$, $r > 0$ and $(R_t)_{t \ge 0}$ be a $\SB(\delta, r)$-process with $0$-hitting time $\sigma$ defined on $(\Omega, \cF,(\cF_t)_{t\ge 0},\PP)$. Then, for any $T, \alpha > 0$:
\begin{enumerate}[label=(\roman*)]
  \item \label{t:BP:scal} $(\frac 1\alpha R_{\alpha s})_{s \ge 0}$ is a $\SB(\delta, \frac r\alpha)$-process,
  \item \label{t:BP:geq2} if $\delta \geq 2$, then $\PP(\sigma < \infty ) = 0$. Moreover, 
  \[
    \PP \Big( \sup_{0 \leq t \leq T} R_t \geq \alpha  \Big) > 0.
  \]
  \item \label{t:BP:less2} if $\delta < 2$, then $\PP(\sigma < \infty) = 1$.
  Moreover, if also $r < \alpha$, then 
  \[
    \PP \Big( \sigma \le T, \ \sup_{0 < t < T} R_t < \alpha \Big) > 0.
  \]
  \item \label{t:BP:time-int} if $\delta < 2$, then 
  $$
    \int_0^\sigma \frac{\dd s}{\sqrt{ R_s }}< \infty. 
    $$
\end{enumerate}
\end{thm}

\begin{proof}
\ref{t:BP:scal} is \cite[Proposition 1.10, p.446]{RevuzYor99}.
\ref{t:BP:geq2} and the first part of \ref{t:BP:less2} are proved on \cite[p.442]{RevuzYor99}. The second part of \ref{t:BP:less2} is established in Step 5 of the proof of \cite[Proposition 4]{FournierJourdain17}. 

Finally, for \ref{t:BP:time-int} we provide the proof in full.
    We apply a time change from $t$ (and $s$) to a new time variable $v$ (and $u$). For all $t\in [0,\sigma)$, set $V_t = \int_0^t R^{-1}_s\dd s$ and for all $v \in [0, V_{\sigma})$, set $T_v = \inf \{ s \ge 0 : V_s \ge v \}$ as the generalized inverse. Setting $Z_v = R_{T_v}$ for all $v \in [0,V_{\sigma})$, we get that for all $v\in [0,V_{\sigma})$,
    \begin{align*}
     Z_v = Z_0 + \int_0^v 2 Z_u \dd B_u + \delta \int_0^v Z_u \dd u, 
    \end{align*}
    where $(B_v)_{v\ge 0}$ is a standard Brownian motion. This SDE has an explicit solution. It is given by
    \begin{align*}
        Z_v = \exp \Big( 2B_v + (\delta -2)v\Big) \quad \mbox{ for all } v\in [0, V_{\sigma}).
    \end{align*}
    Then, performing the change of variable $u = V_s$ we get
    $$
    \int_0^{\sigma} \frac{\dd s}{\sqrt{ R_s} } = \int_0^{V_{\sigma}} \sqrt{ Z_u } \dd u \le \int_0^\infty \exp \Big( B_u + \Big( \frac\delta2 -1 \Big) u\Big) \dd u.
    $$
    The last quantity is a.s.\ finite since a.s.\ $B_u/u \to 0$ as $u\to \infty$.
\end{proof}

\begin{rem} \label{r:BP}
From the same proof it follows that Theorem \ref{t:BP}\ref{t:BP:geq2} also holds for squared Bessel processes that start at $r = 0$ and are not stopped at $0$.
\end{rem}

\section{Strong well-posedness prior to collision}
\label{s:SC}

\subsection{The solution to the SDE system}
\label{s:SC:sol}

In this section, we build a precise solution concept for the SDE system \eqref{SDE} up to the first collision, and show that it has strong existence and uniqueness. We will be pedantic in the statements of our lemmas in Sections \ref{s:SC:sol} and \ref{s:SC:prop}, because we apply them in various settings in the subsequent sections, where they play a crucial role in the proofs. 

Let $(\Omega, \cF)$ be a measurable space. For any probability $\PP$ that we introduce (possibly with sub- and superscripts) on $(\Omega, \cF)$ we denote by either $\E_\P$ or $\E$ (with the same sub- and superscripts) the corresponding expectation.

Recall $\cE$ from \eqref{cE} as the set of particle configurations with distinct positions. For any $\e > 0$, we introduce
\begin{align} \label{cEe} 
    \cE^\e := \Big\{x\in (\R^2)^N : \min_{1\le i\ne j \le N}\|x^i-x^j\| \ge \e \Big\} \subset \cE
\end{align}
as the set of separated particle configurations. Note that $F$ from \eqref{SDE:vec-form} is smooth on $\ov \cE^\e$ for each $\e$.
We also introduce the one-point compactification of $\cE$ (recall \eqref{cE}) given by $\cE \cup \{ \triangle \}$, which can be thought of as that any sequence $(x_k) \subset \cE$ that either tends to $\partial \cE$ or to $\infty$, converges in $\cE \cup \{ \triangle \}$ to the abstract point $\triangle$.
\comm{ChatGPT says r.v. is OK (see below), and r.v.s for the plural.}

\begin{defn}[Strong solution of \eqref{SDE}] \label{d:s-sol}
   Let $\gamma > 0$, $N \geq 1$, $b \in \{-1,1\}^N$ and $(\Omega,\cF,(\cF_t)_{t\ge 0},\P)$
   be a filtered probability space.
   Let $(B_t)_{t\ge 0}$ be a $2N$-dimensional $(\cF_t)_{t\ge 0}$-adapted Brownian motion.
   Let $X_\circ$ be an $\cE\cup \{\triangle\}$-valued $\cF_0$-measurable r.v. Let $(\cF^B_t)_{t\ge 0}$ be the canonical filtration of $(B_t)_{t\ge 0}$. We say that an $\cE\cup \{\triangle\}$-valued process $(X_t)_{t\ge 0}$ is a strong solution of \eqref{SDE} starting from $X_\circ$ if
   \begin{itemize}
       \item $X_t$ is $\sigma(X_\circ, \cF^B_t)$-measurable for all $t\ge 0$, \comm{[For this property we want $\cF^B$ in addition to $\cF$]}
       \item $\displaystyle X_t = X_\circ + B_t + \int_0^t F(X_s)\dd s$ \quad for all $t\in [0,\tau)$, where
       \begin{align} \label{taue} 
    \tau_\e := \inf\{t\ge 0 : X_t \notin \cE^\e\}, \qquad 
    \tau := \lim_{\e \to 0} \tau_\e \ \text{ a.s., and}
\end{align}
       \item $X_t = \triangle$ for all $t\ge \tau$.
   \end{itemize}   
\end{defn}

The limit in \eqref{taue} exists because $\cE^\e \nearrow \cE$ as $\e \to 0$ and therefore $(\tau_\e)_\e$ is monotone.

\begin{lem}\label{l:exun:X}
Let $\gamma > 0$, $N \geq 1$, $b \in \{-1,1\}^N$, and $(\Omega,\cF,(\cF_t)_{t\ge 0},\PP)$ be a filtered probability space on which is defined a $2N$-dimensional $(\cF_t)_{t\ge 0}$-adapted Brownian motion $(B_t)_{t\ge 0}$. For all $x\in \cE \cup \{\triangle\}$,
there exists a unique strong solution $(X_t)_{t\ge 0}$ to \eqref{SDE} with $X_\circ = x$. Moreover, setting $\PP_x$ as the law of $(X_t)_{t\ge 0}$, the tuple $(\Omega, \cF, \PP, (\cF_t)_{t\ge 0}, (X_t)_{t\ge 0},(\PP_x)_{x\in \cE\cup\{\triangle\}})$ is an $\cE\cup\{\triangle\}$-valued strong Markov process which is a.s.\ continuous on $[0,\tau)$.
\end{lem}
\comm{[$\PP_x$ is on $\Omega$, and 'law' and 'distribution' mean the same thing]}

\begin{proof}[Proof of Lemma \ref{l:exun:X}]
As preparation, we consider for every $x\in (\R^2)^N$ and every $\e > 0$ the smoothed SDE
\begin{align} \label{SDE:Xe}
    X^\e_t = x + B_t + \int_0^t F_\e (X^\e_s)\dd s
    \qquad \mbox{for all } t\ge 0,
\end{align}
where $F_\e : (\R^2)^N \to (\R^2)^N$ is a smooth extension of $F|_{\cE^\e}$ to $(\R^2)^N$. 
It is classical (see for example \cite[Chapter 5]{KaratzasShreve98}) that \eqref{SDE:Xe} has a unique strong solution $(X^\e_t)_{t\ge 0}$.
Moreover, for each $\e > 0$, $(\Omega,\cF,\PP,(\cF_t)_{t\ge 0}, (X^\e_t)_{t\ge 0},(\PP^\e_x)_{x\in (\R^2)^N})$ is a $(\R^2)^N$-valued strong Markov process, where for all $x\in (\R^2)^N$, $\PP^\e_x$ is the law of $(X^\e_t)_{t\ge 0}$ starting from $x$, with corresponding expectation $\E^\e_x$. 
 
Next we prove the existence of a strong solution $(X_t)_{t \ge 0}$ for any $x \in \cE \cup \{\triangle\}$ by constructing one. If $x  = \triangle$, then $X_t := \triangle$ for all $t \ge 0$ clearly defines a strong solution. If $x \in \cE$, then, for any $\e$, let $(X^\e_t)_{t\ge 0}$ be as defined from \eqref{SDE:Xe}.
According to $F = F_\e$ on $\cE^\e$, the definition of $\tau_\e$ and the pathwise uniqueness property of the processes $(X^\e_t)_{t\ge 0}$ for all $\e >0$, we have for all $\e>\eta>0$, a.s. \comm{[This interval is open at $\tau_\e$ only because it may be infinite.]}
$$
\mbox{ for all } t\in [0,\tau_\e) \quad X^\e_t = X^{\eta}_t.
$$
Recall that $(\tau_\e)_{\e>0}$ is monotone, and thus the a.s.\ limit $\tau := \lim_{\e \to 0} \tau_\e$ exists. This allows us to define for all $t\in [0,\tau)$, $X_t := \lim_{\e \to 0} X^\e_t$, since this sequence is a.s.\ stationary. Moreover, passing to the limit $\e \to 0$ in \eqref{SDE:Xe}, first for a fixed stopping time $\tau_\eta$, and then $\eta \to 0$, we obtain
$$
\mbox{ for all } t\in [0,\tau), \quad X_t = x + B_t + \int_0^t F (X_s)\dd s.
$$   
We further set $X_t := \triangle$ for all $t\ge \tau$. This shows that $(X_t)_{t \ge 0}$ is a strong solution.

The uniqueness of strong solutions is trivial for $x = \triangle$. For $x \in \cE$, we observe that any strong solution $(X_t)_{t \ge 0}$ satisfies the smoothed SDE \eqref{SDE:Xe} on $[0, \tau_\e)$ for all $\e$. From the uniqueness of $(X_t^\e)_{t \ge 0}$, it follows that $(X_t)_{t \ge 0}$ is unique on $[0, \tau_\e)$. Then, taking $\e \to 0$ as above, it follows that $(X_t)_{t \ge 0}$ is unique on $[0, \tau)$. The uniqueness on $[\tau, \infty)$ is obvious. 

Finally, we prove the strong Markov property of $(\Omega, \cF, \PP, (\cF_t)_{t\ge 0}, (X_t)_{t\ge 0},(\PP_x)_{x\in \cE\cup\{\triangle\}})$. By the strong Markov property of $(\Omega, \cF, \PP, (\cF_t)_{t\ge 0}, (X_t^\e)_{t\ge 0},(\PP_x^\e)_{x\in (\R^2)^N})$ for any $\e$, we have for any $x \in (\R^2)^N$, any $(\cF_t)_{t\ge 0}$-stopping time $\sigma$,
any $t\ge 0$, any $f\in C_c(\cE)$ (with $0$-valued extension to $(\R^2)^N$ and $\triangle$), any $A \in \cF_\sigma$, and any $\e >0$ \comm{[I did the MP wrong before. Lesson: $\sigma, \tau_\e, X_t^\e$ etc are RVs, i.e.\ fcts on $\Omega$. Their distr-ns are fixed wrt $\PP$, but they can ALSO be integrated against other distr-ns, such as $\PP_x^\e$ at various $x$; even to random $x$.]} 
\comm{\quad [$f$ needs to be extended (to some finite value), because the below is supposed to be defined for all $t$.]}
$$
\E_x^\e[\indiq_Af(X^\e_{\sigma+t})\indiq_{\sigma+t<\tau_\e}] 
= \E_x^\e[\indiq_A \indiq_{\sigma<\tau_\e} \E_{X^\e_\sigma}^\e [f(X^\e_t)\indiq_{t<\tau_\e}]].
$$
Since $X_t = X^\e_t$ for all $t\in [0,\tau_\e)$, we directly get that 
$$
\E_x[\indiq_Af(X_{\sigma+t})\indiq_{\sigma+t<\tau_\e}] = \E_x[\indiq_A \indiq_{\sigma <\tau_\e}\E_{X_\sigma} [f(X_t)\indiq_{t<\tau_\e}]],
$$
where $\E_x$ is the expectation with respect to $\P_x$.
Letting $\e\to 0$ and applying the dominated convergence theorem we obtain the result.
\end{proof}

Definition \ref{d:s-sol} and Lemma \ref{l:exun:X} allow for convenient terminology.
Given the setting in Definition \ref{d:s-sol}, we call $(X_t)_{t\ge 0}$ an $\SC(b,X_\circ)$-process and $\tau$ from \eqref{taue} its lifetime. If the Brownian motion is not specified, then we call $(X_t)_{t\ge 0}$ an $\SC(b,X_\circ)$-process if there exists an $(\cF_t)_{t \ge 0}$-adapted Brownian motion $(B_t)_{t \ge 0}$ for which $(X_t)_{t\ge 0}$ is a strong solution, and refer to $(B_t)_{t \ge 0}$ as the associated Brownian motion. If the initial condition is not specified, then we call $(X_t)_{t\ge 0}$ an $\SC(b)$-process if it is a strong solution with $X_\circ = X_0$.

Obviously, $\SC(b,x)$-processes depend on $\gamma$, $N$ and $(\Omega,\cF,(\cF_t)_{t\ge 0},\PP)$ too. We do not adopt this dependence in the notation, because $\gamma$ is kept fixed in all arguments, $N$ follows from the number of coordinates in $b$, and we only specify $(\Omega, \cF, (\cF_t)_{t\ge 0}, \PP)$ whenever this level of detail is helpful for understanding theorems and proofs.

\subsection{Properties of $\SC(b)$-processes}
\label{s:SC:prop}

In this section we build the five tools of $\SC(b)$-processes $(X_t)_{t \ge 0}$ that we mentioned in Section \ref{s:intro:proof}. These tools are given by Proposition \ref{p:invar}\ref{p:invar:scale}, Lemma \ref{l:R:sBP} and Lemma \ref{l:Girsanov}. 
 
First, we collect some simple observations on $\SC(b)$-processes \cite[Section 2.1 and Proposition 2.2]{VanMeursPeletierSlangen25}:
\begin{prop}[Invariances of $\SC(b)$-processes] \label{p:invar}
Let $\gamma > 0$, $N \geq 1$, $b \in \{-1,1\}^N$ and \\ $(\Omega,\cF,(\cF_t)_{t\ge 0},\PP)$ be a filtered probability space on which is defined a standard Brownian motion $(B_t)_{t\ge 0}$ and an associated $\SC(b)$-process $(X_t)_{t\ge 0}$ with lifetime $\tau$. Then:
\begin{enumerate}[label=(\roman*)]
  \item \label{p:invar:b} (invariance with respect to flipping signs) $(X_t)_{t\ge 0}$ is also a $\SC(-b)$-process with respect to the same Brownian motion $(B_t)_{t\ge 0}$,
  \item \label{p:invar:label} (invariance with respect to relabelling) Let $\sigma$ be a permutation on $\bbrack{1,N}$. Let $b^\sigma := (b^{\sigma(i)})_{i=1}^N$ and define similarly $X_t^\sigma$ and $B_t^\sigma$.
  Then, $(X_t^\sigma)_{t\ge 0}$ is an $\SC(b^\sigma)$-process with respect to the Brownian motion $(B_t^\sigma)_{t\ge 0}$,
  \item \label{p:invar:scale} (scale invariance) For any finite $\cF_0$-measurable r.v.\ $L > 0$, the process $(\tilde X_s)_{s \ge 0}$ defined for all $s\in [0,\tau/L^2)$ by $\tilde X_s := \frac1L X_{L^2 s}$ and for all $s\ge \tau/L^2$ by $\tilde X_s := \triangle$, is an $\SC(b)$-process generated by the rescaled Brownian motion $(L^{-1}B_{s L^2})_{s \ge 0}$. 
  \comm{[1.\ Splitting two cases of $s$ is needed cause $\triangle$ is not divisible (by $L$). In general, it is better to treat all $\triangle$ stuff separately. 2.\ Proof: use $F(\alpha x) = \frac1\alpha F(x)$ and a simple application of the time-change formula for SDEs]}
\end{enumerate}
\end{prop}

The results in the remainder of this section are inspired by \cite[Section 5]{Tardy24} and \cite{FournierTardy25}. There, similar statements are proved for the Keller--Segel system \eqref{KS}. It turns out that these proofs apply with relatively minor modifications to $\SC$-processes. Nevertheless, we present the proofs in full detail, and reformulate the statements for later use.

First, we examine the processes $(M(X_t))_{t \ge 0}$ and $(R(X_t))_{t \ge 0}$, which we recall from \eqref{RK} to be the local empirical mean $(M^K(X_t))_{t \ge 0}$ and the local dispersion $(R^K(X_t))_{t \ge 0}$ with $K = \bbrack{1,N}$. Since $X_t$ may equal the abstract one-point compactification point $\triangle$ of $\cE$, the definitions of these processes require an extension of the functions $M^K$ and $R^K$. For all $K \subset \bbrack{1,N}$ we define $M^K(\triangle) := \triangle$ and $R^K(\triangle) := \triangle$ with the following understanding: these values are abstract points different from the one-point compactification of $\cE$ such that $(M^K(X_t))_{t \ge 0}$ and $(R^K(X_t))_{t \ge 0}$ are always discontinuous at $t = \tau$ on the event $\{\tau < \infty\}$.

In the deterministic setting, the evolutions of $(M(X_t))_{t \ge 0}$ and $(R(X_t))_{t \ge 0}$ are very simple (recall \eqref{pfxd}). 
The following lemma (counterpart of \cite[Lemma 5.2]{Tardy24}) treats the stochastic case. It states that, roughly speaking, for an $\SC(b)$-process $(X_t)_{t \ge 0}$ with lifetime $\tau$, $(M(X_t))_{t \ge 0}$ is a $2$-dimensional Brownian motion scaled by $N^{-1/2}$ and killed at $\tau$, $(R(X_t))_{t \ge 0}$ is a $\SB(\delta)$-process killed at $\tau$, where we recall from \eqref{delta} and \eqref{deltaK} that
\[
 \mbox{ for all } K\subset \{ 1,\dots, N\}, \quad \delta_K = \gamma \bigg( \Big( \sum_{i \in K} b^i \Big)^2 - |K| \bigg) + 2(|K|-1) 
 \quad \mbox{ and }\quad 
  \delta = \delta_{\bbrack{1,N}},
\]
and that both processes are independent up to $\tau$ knowing the initial condition. The rough part of this statement is the `independence up to $\tau$'. 
 
\begin{lem}[Properties of $(M(X_t))_{t \ge 0}$ and $(R(X_t))_{t \ge 0}$] \label{l:R:sBP}
Let $\gamma > 0$, $N \geq 1$, $b \in \{-1,1\}^N$, $x \in \cE$, and $(\Omega,\cF,(\cF_t)_{t\ge 0},\PP)$ be a filtered probability space on which is defined a standard Brownian motion $(B_t)_{t\ge 0}$ and an associated $\SC(b,x)$-process $(X_t)_{t\ge 0}$ with lifetime $\tau$.
Then, even if it means to enlarge the probability space,
there exists two independent processes $(S_t)_{t\ge 0}$ and $(Z_t)_{t\ge 0}$ such that
\begin{itemize}
    \item $(S_t)_{t\ge 0}$ is a $2$-dimensional Brownian motion scaled by $N^{-1/2}$ and started at $M(x)$,
    \item $(Z_t)_{t\ge 0}$ is $\SB(\delta, R(x))$-process (recall that for $\delta < 0$, $(Z_t)_{t\ge 0}$ is stopped at $0$),
    \item for all $t\in [0,\tau)$, $M(X_t) = S_t$ and $R(X_t) = Z_t$.
\end{itemize} 
\end{lem}

\begin{proof}
First, writing the SDE in Definition \ref{d:s-sol} component-wise and then summing the components, we obtain (recalling that $M(F) = 0$) $M(X_t) = M(x) + M(B_t)$ for all $t \in [0,\tau)$. Thus, taking $S_t := M(x) + M(B_t)$ for all $t \ge 0$, we get both that $M(X_t) = S_t$ on $[0,\tau)$ and that $(S_t)_{t \ge 0}$ is as stated. 

\vspace{0.3cm}

The case of $R$ requires more care; we cannot directly apply It\^o's lemma to $R(X_t)$ because of the singularity in $F$ and the stopping time $\tau$. Instead, we fix $\e > 0$, and recall the smoothed SDE \eqref{SDE:Xe} with global solution $(X_t^\e)_{t \ge 0}$. Then, It\^o's lemma yields (with $\nabla_{y^i} R(y) = 2(y^i - M(y))$ and $\Delta R(y) = 4(N-1)$)
\begin{multline*}
    R(X_t^\e) = R(x) + \int_0^t \Big( 2(N-1) + \sum_{i=1}^N 2 (X_s^{\e,i} - M(X_s^\e)) \cdot F_\e^i(X_s^\e) \Big) d s  \\
    + \sum_{i=1}^N \int_0^t 2 (X_s^{\e,i} - M(X_s^\e)) \cdot d B_s^i
    \qquad \text{for all } t \ge 0.
\end{multline*}
Recall $\tau_\e$ from \eqref{taue},
and note that, by construction, for all $t \in [0, \tau_\e) $,  $X_t^\e = X_t$ and $F_\e(X_t^\e) = F(X_t) $. Then, from the observations in Section \ref{s:intro:proof} that $M(F) = 0$ and $y \cdot F(y)$ is constant in $y \in \cE$, it follows from a simple computation that (for details, see \cite[(4.26)]{VanMeursPeletierSlangen25})
\begin{align} \label{pfzy} 
  R(X_t) = R(x) + 2 \int_0^t \sqrt{R(X_s)} d \beta_s + \delta t
   \qquad \text{for all } t \in [0, \tau^\e),
\end{align} 
where
\begin{equation*}
  \beta_t := \sum_{i = 1}^N \int_0^t \frac{  (X_s^i - M(X_s)) }{\sqrt{ \sum_{j = 1}^N |X_s^j - M(X_s)|^2 }} \cdot d B_s^i
  =: \int_0^t Y(X_s) \cdot d B_s \qquad \text{for all } t \in [0, \tau).
\end{equation*}
Note that the square of the denominator equals $R(X_s) = R(X_s^\e)$, which is bounded from below by $\frac12 (N-1) \e^2 > 0$ on $[0, \tau^\e)$.
It follows from the L\'evy characterization and $\|Y(X_s)\| = 1$ that $(\beta_t)_{t\ge 0}$ is a $1$-dimensional Brownian Motion on $[0,\tau^\e)$. 

Since $\e > 0$ was arbitrary and $\tau^\e \nearrow \tau$ a.s., we obtain that $(\beta_t)_{t\ge 0}$ is a $1$-dimensional Brownian Motion on $[0,\tau)$. Moreover, \eqref{pfzy} holds on $[0,\tau)$. \comm{[To verify this, go ptws in $\omega$]}

Next we construct $(Z_t)_{t \ge 0}$. 
Even if it means to enlarge the probability space, there exists a standard Brownian motion $(W_t)_{t\ge 0}$ independent from $(B_t)_{t\ge 0}$. We set for all $t\ge 0$, 
\begin{equation} \label{pfxj}
   \tilde \beta_t = \indiq_{t<\tau}\beta_t + \indiq_{t\ge \tau}( W_t -W_\tau + \beta_\tau). 
\end{equation}
It is classical that $(\tilde \beta_t)_{t\ge 0}$ is a standard Brownian motion. Let $(Z_t)_{t\ge 0}$ be the $\SB(\delta, R(x))$-process driven by $(\tilde \beta_t)_{t\ge 0}$ with $0$-hitting time $\sigma$ (recall Section \ref{s:BP}). It is clear from \eqref{pfzy} that $Z_t = R(X_t)$ on $[0, \tau \wedge \sigma) = [0, \tau)$. 

\vspace{0.3cm} Finally, it remains to show that $(S_t)_{t\ge 0}$ and $(Z_t)_{t\ge 0}$ are independent processes. This is clear since both SDEs are well-posed (see \cite{RevuzYor99}), $Z_t$ is $(\tilde\beta_t)_{t\ge 0}$-measurable, and $(\tilde \beta_t)_{t\ge 0}$ and $(S_t)_{t\ge 0}$ are independent according to the L\'evy characterization since for all $t\ge 0$ and for all $k\in \{1,2\}$,
\comm{[comput on p.69 2/3]}
$$
\langle \tilde \beta, S \cdot e_k \rangle _t = \int_0^{t\wedge\tau } \sum_{i=1}^N  \frac{(X^{i}_s - S_s) \cdot e_k}{N \sqrt{Z_s}}\dd s = 0,
$$
where $(e_1,e_2)$ is the canonical base of $\R^2$. The existence of the integral is guaranteed by Theorem \ref{t:BP}\ref{t:BP:time-int}.
\end{proof}  

Following Section \ref{s:ODE:ex}, we now examine the processes $(M^K(X_t))_{t \ge 0}$ and $(R^K(X_t))_{t \ge 0}$ for any $K \subset \bbrack{1,N}$ with $1 \le |K| \le N-1$. The strategy is to select disjoint random time intervals $[ \sigma^k, \tilde \sigma^k)$ indexed by $k \ge 0$ on which the particle configuration $X_t$ can be split into sufficiently separated clusters, and to show that each cluster behaves like an independent $\SC$-process with suitable parameters through a Girsanov argument. To make this precise, we introduce notation.
First,
for nonempty $K \subset \bbrack{1,N}$ we introduce similarly to $x^K$ in \eqref{xK}
\begin{align*} 
    b^K := (b^i)_{i \in K} \in \{-1,1\}^{|K|}.
\end{align*}
Note that
\begin{align*} 
  M(x^K) = M^K(x), \quad 
  R(x^K) = R^K(x),
\end{align*} 
with the understanding that `$N$' in the definitions of $M,R$ in Section \ref{s:intro:proof} is $|K|$ as determined by the input.
Second, for any $K\subset \ig 1,N\id$ with $1 \le |K| \le N-1$, we set 
\begin{align}\label{dK}
    \mbox{for all } x\in (\R^2)^N, \quad d^K(x) := \min_{i\in K, j\notin K} \| x^i-x^j\|
\end{align}
as the cluster separation distance between $x^K$ and $x^{K^c}$.
Third, let $\bK$ be a partition of $\ig 1,N\id$ with $|\bK| \geq 2$, where $|\bK|$ is the number of elements of $\bK$. We introduce the minimal cluster separation distance as
     \begin{equation} \label{dbK}
     d^\bK : (\R^2)^N \to [0,\infty), \qquad 
        d^\bK(x) 
        := \min_{K \in \bK} d^K(x).
     \end{equation}
In the special case $\bK = \{K, K^c\}$, we have $d^\bK = d^K = d^{K^c}$; hence $d^\bK$ generalizes $d^K$.    
Fourth, for any $\e >0$, we introduce the set (recall $\cE^\e$ from \eqref{cEe})
$$
\cE^\e_\bK 
:= \big\{ x\in \cE : d^\bK(x) \ge \e \big\}
\supset \cE^\e
$$
with corresponding exit time
$$
\sigma^\e_\bK := \inf \{ t\ge 0 : X_t \notin \cE^\e_\bK\}.
$$
Note that, since $\triangle \notin \cE^\e_\bK$, we have $\sigma^\e_\bK\le \tau$.
Also, for the trivial partition $\bK$ defined by $|\bK| = N$, i.e.\ $\bK$ consists of $N$ singletons, we have $\cE^\e_{\bK} = \cE^\e$.

\begin{lem}[Girsanov argument for the particle clusters $(X_t^K)_{t \ge 0}$] \label{l:Girsanov}
Let $\gamma, \e, T > 0$, $N \geq 1$, $b \in \{-1,1\}^N$, $x \in \cE$, $\bK$ be a partition  of $\bbrack{1,N}$, and $(\Omega,\cF,(\cF_t)_{t\ge 0},\P)$ be a filtered probability space on which is defined a standard $(\cF_t)_{t\ge 0}$-adapted Brownian motion $(B_t)_{t\ge 0}$. Let $(X_t)_{t\ge 0}$ be the $\SC(b,x)$-process with lifetime $\tau$ driven by $(B_t)_{t\ge 0}$. Let $(\sigma^k)_{k\ge 0}$ and $(\tilde \sigma^k)_{k\ge 0}$ be sequences of $(\cF_t)_{t\ge 0}$-stopping times such that a.s.
\[
  \mbox{for all } k\ge 0,
  \quad \sigma^k \leq \tilde \sigma^k \leq \sigma^{k+1} \leq \tau \wedge T,
\]
and a.s.\ (recall $d^\bK$ from \eqref{dbK})
$$
\mbox{for all } k\ge 0, \mbox{ all } t\in [\sigma^k,\tilde \sigma^k), \quad d^\bK(X_t) \ge \e.
$$
Then, even if it means to enlarge the probability space, there exists a probability  
$\Q$ and a constant $C = C(\gamma, \e, T, N) > 0$ with $\frac1C \le \E_\P[(\dd \Q/\dd \P)^2] \le C$ such that, under $\Q$, there exist $\SC(b^K, X_{\sigma^k}^K)$-processes $(\hat X^{K,k}_s)_{s\ge 0}$ indexed by $K\in \bK$ and $k \geq 0$ such that 
\begin{enumerate}[label = {(\roman*)}]
    \item \label{l:Girsanov:indep} the Brownian motions associated to $(\hat X^{K,k}_s)_{s\ge 0}$ for each $K \in \bK$ and each $k \ge 0$ are independent,
    \item \label{l:Girsanov:SCproc} for each $k \ge 0$, each $K \in \bK$ and all $s \in [0, \tilde \sigma^k - \sigma^k)$, $X^{K}_{s + \sigma^k}
   = \hat X^{K,k}_s$.
\end{enumerate}
\end{lem}

\begin{proof}[Proof of Lemma \ref{l:Girsanov}] The proof is a modification of that of \cite[Lemma 5.3]{Tardy24} and is also inspired from \cite{FournierTardy25}.
    For each $K \in \bK$ and each $i\in K$, writting $F^i = F_{K}^i + F_{K^c}^i$ with
    \[
      F_{K'}^i(y) := \gamma \sum_{j \in K' \setminus \{i\}}  b^i b^j f(y^i - y^j) \qquad \text{for any } K' \subset \bbrack{1,N},
    \]
    and then partitioning the time interval according to $\sigma^k$ and $\tilde \sigma^k$, we obtain (setting $\tilde \sigma^{-1} := 0$)
\begin{equation} \label{pfxk}
  X_t^i 
  = x^i 
    + \bar B_t^i 
    +  \int_0^t F_{K}^i(X_s) \, ds
    + \int_0^t \Big( \sum_{k=0}^\infty \indiq_{[\tilde \sigma^{k-1}, \sigma^k)} (s) \Big) F_{K^c}^i(X_s) \, ds,
\end{equation}
for all $t \in [0, \tau \wedge T)$, where
\begin{gather*}
  \bar B_t^i := B_t^i + \int_0^t \Big( \sum_{k=0}^\infty \indiq_{[\sigma^k, \tilde \sigma^k)} (s) \Big) F_{K^c}^i(X_s) ds
  \qquad \text{for all } t \ge 0.
\end{gather*}
In preparation for applying Girsanov's theorem to $(\bar B_t)_{t \in [0,T]}$, we define for all $t \in [0,T]$
    \begin{align*}
    L_t := \sum_{K \in \bK} \sum_{i \in K} \int_0^t \Big( \sum_{k=0}^\infty \indiq_{[\sigma^k, \tilde \sigma^k)} (s) \Big) F_{K^c}^i(X_s) \cdot d B_s^i.
    \end{align*}
Since for all $K \in \bK$, all $i \in K$, all $j \in K^c$, all $k \geq 0$ and all $s \in [\sigma^k, \tilde \sigma^k)$ we have $|f(X_s^i - X_s^j)| \leq 1/d^\bK(X_s) \le \frac1\e$, we obtain by the Cauchy-Schwarz inequality
    \begin{align} \notag 
        \langle L\rangle_T 
        &= \sum_{k=0}^\infty \int_{\sigma^k}^{\tilde \sigma^k} \sum_{i \in K} \Big|\sum_{K \in \bK}   F_{K^c}^i (X_s) \Big|^2 ds \\ \label{pfxp}
        &\le TN \sum_{K \in \bK} \sum_{i \in K} \gamma^2  |K^c|^2 \frac1{\e^2} < \frac{\gamma^2 T N^4}{\e^2} =: C.
    \end{align}
    Then, using the Novikov criterion,
    we get from Girsanov's theorem that $\dd \Q := \cM(L_T)\dd \P$, with $\cM(L_T) := \exp ( - L_T - \frac12 \langle L\rangle_T )$, is a probability measure under which $(\bar B_t)_{t \ge 0}$ is a $2N$-dimensional Brownian motion. 
    Moreover, we have \comm{[By construction of $\Q$, $\bar B_t$ remains a BM after $T$]}
    $$
    \E_\P[(\dd \Q/\dd \P)^2] = \E_\P [ \exp (-2L_T - \langle L \rangle_T)] = \E_\P [ \exp (-2L_T - 2\langle L \rangle_T) \exp(\langle L\rangle_T)],
    $$
    and the desired bound follows from \eqref{pfxp} and $\E_\P [ \exp (-2L_T - 2\langle L \rangle_T)] = 1$ according to the Novikov criterion.
    
    Reflecting back on \eqref{pfxk}, we get that for any $K \in \bK$, any $i\in K$ and any $k \ge 0$,
    \begin{equation*} 
  X_t^i 
  = X_{\sigma^k}^i 
    + \bar B_t^i - \bar B_{\sigma^k}^i 
    +  \int_{\sigma^k}^t F_{K}^i(X_s) \, ds,
    \qquad \text{for all } t \in [\sigma^k, \tilde \sigma^k).
\end{equation*}
By shifting time as $\tilde X_s^k := X_{s + \sigma^k}$ and setting $\tilde B_s^k := \bar B_{s + \sigma^k} - \bar B_{\sigma^k}$ for all $s \ge 0$, we obtain that $(\tilde B_s^k)_{s \ge 0}$ is a $\Q$-Brownian motion on $[0, \tilde \sigma^k - \sigma^k)$ and that for each $K\in \bK$,
\begin{equation*} 
  \tilde X_s^{k,i} 
  = X_{\sigma^k}^i 
    + \tilde B_s^{k,i} 
    +  \int_0^s F_{K}^i(\tilde X_r^k) \, dr
    \qquad \text{for all } s \in [0, \tilde \sigma^k - \sigma^k) \text{ and each } i \in K.
\end{equation*}
To construct from $\{ (\tilde B_s^k)_{s \ge 0} \}_{k \ge 0}$ a family of independent Brownian motions,
we extend $(\Omega, \cF, \Q)$ if needed to accommodate a family $((W_t^k)_{t \ge 0})_{k \ge 0}$ of standard Brownian motions in $(\R^2)^N$ that are $\Q$-independent of the other random variables previously defined. 
Similar to \eqref{pfxj}, we define for all $k\ge 0$ the $\Q$-Brownian motion $(\hat B_s^k)_{s \ge 0}$ as \comm{["...independent of anything else" is vague. "...independent of the other random variables previously defined" is good.]}
\[
  \hat B_s^k = \indiq_{s<\tilde\sigma^k-\sigma^k} \tilde B_s^k + \indiq_{s\ge \tilde\sigma^k-\sigma^k}( W_s^k -W_{\tilde\sigma^k-\sigma^k}^k + \tilde B_{\tilde\sigma^k-\sigma^k}^k).
\]
Since the intervals $\{ [\sigma^k,\tilde \sigma^k) \}_{k \ge 0}$ are disjoint, $\{ (\hat B_s^k)_{s \ge 0} \}_{k \ge 0}$ is independent.
For each $k \ge 0$, let $(\hat X_s^{k,K})$ be the $\Q$-$\SC(b^K, X_{\sigma^k}^K)$-process generated by $(\hat B_s^k)_{s \ge 0}$ with lifetime $\tau_k^K$. Then, Lemma \ref{l:Girsanov}\ref{l:Girsanov:indep} is satisfied, and
\begin{equation*}
    \hat X_s^{k,i}  
  = X_{\sigma^k}^i 
    + \hat B_s^{k,i} 
    +  \int_0^s F_{K}^i(\hat X_r^k) \, dr
    \qquad \text{for all } s \in [0, \tau_k^K) \text{ and each } i \in K
\end{equation*}
for each $K \in \bK$. Since $\hat B_s^k = \tilde B_s^k$ for all $s \in [0, \tilde \sigma^k - \sigma^k)$, we have by pathwise uniqueness (see Lemma \ref{l:exun:X}) that $\hat X_s^{k,K} = \tilde X_s^{k,K} = X_{s + \sigma^k}^K$ for all $K \in \bK$, all $k \ge 0$ and all $s \in [0, (\tilde \sigma^k - \sigma^k) \wedge \tau_k^K)$. By definition of the lifetime, we obtain that $\tau_k^K \ge \tilde \sigma^k - \sigma^k$, and thus \ref{l:Girsanov:SCproc} is proven.
\end{proof}

Finally, we state a remark and a corollary of Lemmas \ref{l:R:sBP} and \ref{l:Girsanov}. Both are a simplification of Lemma \ref{l:Girsanov} with Lemma \ref{l:R:sBP} applied to the provided processes $(\hat X^{K,k}_s)_{s\ge 0}$, which are all built on the same probability space $(\Omega, \cF, \Q)$. The sole purpose of these two statements is to provide convenient references for the subsequent parts of this paper.

\begin{rem} \label{r:Girs:K}
Let the setting of Lemma \ref{l:Girsanov} be given with $\bK = \{K, K^c\}$ for some index set $K$ with $1 \le |K| \le N-1$. Then, there exists a probability $\Q$ (possibly after enlarging the probability space) equivalent to $\P$ under which there exist $\SB(\delta_K)$-processes $(Z_s^k)_{s\ge 0}$ indexed by $k \in \N$ such that each is generated by an independent Brownian motion and $R^K(X_{s + \sigma^k}) = Z_s^k$ for all $k \in \N$ and all $s \in [0, \tilde \sigma^k - \sigma^k)$.
\end{rem}

In the following corollary we limit 
$(\sigma^k)_{k\ge 0}$ and $(\tilde \sigma^k)_{k\ge 0}$ to a single stopping time $\sigma$ (which fits to the general setting by taking $\sigma^0 = 0$, $\tilde \sigma^0 = \sigma$, and $\sigma^k = \tilde \sigma^k = \tau \wedge T$ for all $k \ge 1$). Then, the processes $(\hat X_s^{K,0})_{s \ge 0}$ constructed in Lemma \ref{l:Girsanov} are independent. 

\begin{cor} \label{c:Girs:bK}
Let $\gamma, \e, T > 0$, $N \geq 1$, $b \in \{-1,1\}^N$, $x \in \cE$, $\bK$ be a partition  of $\bbrack{1,N}$, and $(\Omega,\cF,(\cF_t)_{t\ge 0},\P)$ be a filtered probability space on which is defined a standard Brownian motion $(B_t)_{t\ge 0}$. Let $(X_t)_{t\ge 0}$ be the $\SC(b,x)$-process with lifetime $\tau$ driven by $(B_t)_{t\ge 0}$. 
Then, even if it means to enlarge the probability space, there exists a probability $\Q$ and a constant $C = C(\gamma, \e, T, N) > 0$ with $\frac1C \le \E_\P[(\dd \Q/\dd \P)^2] \le C$ such that, under $\Q$, there exist processes $(S_t^K)_{t\ge 0}$ and $(Z_t^K)_{t\ge 0}$ indexed by $K \in \bK$ such that
\begin{itemize}
    \item for all $K \in \bK$, $(S_t^K)_{t\ge 0}$ is a $2$-dimensional Brownian motion scaled by $|K|^{-1}$ and started at $M(x)$,
    \item for all $K \in \bK$, $(Z_t^K)_{t\ge 0}$ is a $\SB(\delta_K, R(x))$-process,
    \item $\{ (S_t^K)_{t\ge 0}, (Z_t^K)_{t\ge 0} \}_{K \in \bK}$ is a family of independent processes, 
    \item for all $K \in \bK$ and all $t\in [0,\sigma)$, $M^K(X_t) = S_t^K$ and $R^K(X_t) = Z_t^K$, where \comm{[The half open $[0,\sigma)$, because even though $\sigma$ is bdd, it could be equal to $\tau$, but then we have defined $X_\tau = \triangle$ on which $M^K$ is not defined.]}
\[
  \sigma 
  := T \wedge \sigma_\bK^\e 
  = T \wedge \tau \wedge \inf \{ t \ge 0 : d^\bK(X_t) < \e \}.
\]    
\end{itemize} 
\end{cor}

\section{Left-continuity at collision}
\label{s:SDE:cty}

This subsection is the stochastic counterpart of Section \ref{s:ODE:ex}; the goal is to prove the following theorem, which is the stochastic counterpart of Proposition \ref{p:ODE}. 

\begin{prop}[Continuity at collision] \label{p:cty:tau}
    Let $\gamma, T > 0$, $N \geq 1$, $b \in \{-1,1\}^N$, $x \in \cE$, and
    $(X_t)_{t\ge 0}$ be an $\SC(b,x)$-process with lifetime $\tau$. Then, the limit $X_{(\tau \wedge T)-} := \lim_{t \nearrow \tau\wedge T} X_t$ exists in $(\R^2)^N$. Moreover, on the event $\{ \tau < \infty \}$, $X_{\tau-} \in \partial \cE$.
\end{prop}

The proof requires some preliminaries, and is therefore postponed to the end of this section.

Similar to Section \ref{s:ODE:ex}, we examine the processes $(M^K(X_t))_{t \ge 0}$ and $(R^K(X_t))_{t \ge 0}$ for any non-empty $K \subset \bbrack{1,N}$. 
We recall that for $t \geq \tau$ we have  $X_t = \triangle$ and thus, by definition, $M^K(X_t) = \triangle \notin \R^2$ and $R^K(X_t) = \triangle \notin \R$.
The following properties hold for the same reason as in the deterministic setting: for all $t \in [0,\tau)$, $M^{\{i\}}(X_t) = X_t^i$, $R^{\{i\}}(X_t) = 0$ and the ordering of $R^K(X_t)$ in $K$ in \eqref{pfzt} holds. 
Yet, the stochastic counterpart of Lemma \ref{l:ODE:ex} requires more care. It is a modification of \cite[Proposition 5.4]{Tardy24}.
 
\begin{prop}[Stochastic counterpart of Lemma \ref{l:ODE:ex}] \label{p:R:alte}
Given the setting from Proposition \ref{p:cty:tau}, for all nonempty $K \subset \bbrack{1,N}$
\begin{equation*}
  \liminf_{t \nearrow \tau \wedge T} R^K(X_t) > 0 \text{ a.s.}
  \quad \text{or} \quad
  \lim_{t \nearrow \tau \wedge T} R^K(X_t) = 0 \text{ a.s.}
\end{equation*}
\end{prop}  

\begin{proof}
Let $(\Omega,\cF,(\cF_t)_{t\ge 0},\P)$ be a sufficiently large underlying filtered probability space and $(B_t)_{t \ge 0}$ be a Brownian motion associated with $(X_t)_{t \ge 0}$.
Recalling that $R^{\{i\}} \equiv 0$, the assertion follows for $|K|=1$. For $K = \bbrack{1,N}$, we have by Lemma \ref{l:R:sBP} that $(R^K(X_t))_{t\ge 0}$ is uniformly continuous on $[0, \tau\wedge T)$, and thus the assertion follows.

Thus, we may assume that $N \geq 3$ and prove the assertion for $2 \leq |K| \leq N-1$. We reason by backward induction by showing that $\cP(n) =$ ``for all $K\subset \ig 1,N\id$ with $|K| \ge n$, we have
\begin{equation*}
  \liminf_{t \to \tau\wedge T} R^K(X_t) = 0 \text{ a.s.}
  \implies
  \limsup_{t \to \tau\wedge T} R^K(X_t) = 0 \text{ a.s."}
\end{equation*}
is true for all $n\in \ig2,N\id$. Above we have proved that $\cP(N)$ is true.

Suppose $\cP(n)$ holds for some $n \geq 3$. Let $K \subset \bbrack{1,N}$ with $|K| = n-1$. Let
\begin{equation*}
  A := \Big\{ \liminf_{t \nearrow \tau \wedge T} R^K(X_t) = 0, \ \limsup_{t \nearrow \tau \wedge T} R^K(X_t) > 0 \Big\},
\end{equation*}
and assume by contradiction that $\P(A) > 0$. 

On the event $A$, we also have 
\begin{equation*}
  \liminf_{t \to \tau\wedge T} \min_{j \in K^c} R^{K^j}(X_t) > 0, \qquad \mbox{ with } \quad K^j := K \cup \{j\},
\end{equation*}
because the induction hypothesis holds for $K^j$ and
\begin{equation*}
  0
  < \limsup_{t \nearrow \tau \wedge T} R^{K}(X_t)
  \leq \frac k{k-1} \limsup_{t \nearrow \tau \wedge T} \min_{j \in K^c} R^{K^j}(X_t).
\end{equation*}

Next, we will construct an event $B \subset A$ with $\PP(B) > 0$ by reducing $A$ step by step. First, by removing a nullset, we may further assume on $A$ that $(X_t)_{t\ge 0}$ is continuous on $[0,\tau)$ for the usual topology on $(\R^2)^N$.
Since on $A$, $(R^K(X_t))_{t\ge 0}$ is not continuous at $\tau\wedge T$, $(X_t)_{t\ge 0}$ is not continuous at $\tau\wedge T$, and thus $\tau\wedge T = \tau$ (i.e.\ $\tau \leq T$). Second, there exist deterministic constants $\alpha, u > 0$  
such that $\P(A \cap B_{\alpha,u})>0$ where for all $\alpha', u' \in \Q_{>0}$,
\comm{["for the usual topology on $(\R^2)^N$": to stress that we don't consider $\cE_\triangle$]}
\begin{multline*}
B_{\alpha',u'} = \Big\{ (R^K(X_t))_{t\ge 0}  \mbox{ crosses } \Big[ \frac{\alpha'}2, \alpha' \Big] \mbox{ infinitely many during } [\tau-u',\tau) \\
\mbox{ and } \min_{j \in K^c} \inf_{t \in [\tau - u', \tau)} R^{K^j}(X_t) \geq 7 \alpha' \Big\} 
\end{multline*}
Indeed, we have $A\subset \bigcup_{\alpha',u' \in \Q_{>0}} B_{\alpha',u'}$ and $\P (A) >0$. In conclusion, we have $\P(B) > 0$ with
\comm{[The reasoning here is that if $A = \cup_k B_k$ for $B_k \subset A$, then $\PP(A) > 0$ implies that $\PP(B_k) > 0$ for at least one $k$, and that is enough.]} 
\begin{equation*}
    B := A \cap B_{\alpha,u} \cap \{ (X_t)_{t\ge 0} \text{ is continuous on } [0,\tau) \}.
\end{equation*}

\vspace{0.3cm}

The computation leading to \eqref{pfzp} shows that on $B$,
\begin{equation} \label{pfzn}
[\tau - u, \tau) \cap \{ t\ge 0 : R^K(X_t) \leq \alpha \}  \subset   \{ t\ge 0 : d^K(X_t) \geq \alpha \} .
\end{equation}
We define the following sequences of stopping time $(\sigma^k)_{k\ge 0}$ and $(\tilde \sigma^k)_{k\ge 0}$ recursively by (setting $\tilde \sigma^{-1} = 0$ and iterating $k = 0,1,2,\ldots$) 
\begin{align*}
  \sigma^k
  &= \tau \wedge T \wedge \inf \big\{ t \geq \tilde \sigma^{k-1} : R^K(X_t) \leq \frac{\alpha}{2} \ \text{ and } \ d^K(X_t) \geq \alpha \big\}, \\
  \tilde \sigma^k
  &= \tau \wedge T \wedge \inf \Big\{ t \geq \sigma^k : R^K(X_t) \geq \alpha \ \text{ or } \ d^K(X_t) \leq \frac{\alpha}{2} \Big\}.
\end{align*} 
By construction, $\sigma^k \le \tilde \sigma^k \le \sigma^{k+1} \le \tau \wedge T$ for all $k \ge 0$.
Let $\Q$ and $((Z^k_s)_{s\ge 0})_{k\ge 0}$ be as in Remark \ref{r:Girs:K} applied with $\gamma$, $T$, $\e = \alpha/2$, $N$, $b$, $x$, $K$, $(\Omega,\cF,\P)$, $(B_t)_{t\ge 0}$, $(X_t)_{t\ge 0}$, $(\sigma^k)_{k\ge 0}$ and $(\tilde \sigma^k)_{k\ge 0}$. 
We set 
\[
\eta^k := \inf\{s \ge 0 : R^K(X_{s + \sigma^k}) \ge \alpha \}
\quad
\mbox{ for all } k\ge 0. 
\]
Then, on $B$, we obtain from \eqref{pfzn} that for all $k\ge 0$, 
$\tilde \sigma^k < \tau\wedge T$. Moreover, there exists a random $k_0\in \N$ such that $\sigma_k \geq \tau - u$ for all $k \ge k_0$. 
\comm{[This should be seen ptws-ly in $\omega$. A uniform-in-$\omega$ version is also possible; then we should add $\1_{k \ge k_0}$ to the RHS, and then it works for all $k\ge 0$.]}
Then, again from \eqref{pfzn}, for all $k \ge k_0$, we have $\eta^k = \tilde \sigma^k - \sigma^k$, and thus 
\begin{equation}\label{pfxs}
  \eta^k = \inf\{t\ge 0 : Z_t^k \ge \alpha \}    
\end{equation}
and 
\begin{equation}\label{pfxr}
  \sum_{k\ge k_0} \eta_k \le u < \infty.    
\end{equation}

On the other hand, under $\Q$, by the comparison theorem (see e.g.\ \cite[Theorem 3.7, p394]{RevuzYor99}) and $\{(Z^k_t)_{t\ge 0}\}_{k\ge 0}$ being generated by independent Brownian motions, 
for all $k\ge 0$ there exists a $\SB(\delta_K, \alpha/2)$-process $(\tilde Z^k_t)_{t\ge 0}$ such that the processes $(\tilde Z^k_t)_{t\ge 0}$ with $k\ge 0$ are independent 
and $\tilde Z^k_t \ge Z^k_t$ for all $t\ge 0$ and each $k \ge 0$.
Then, noting that the r.v.s $\tilde \eta^k := \inf \{ t\ge 0 : \tilde Z^k_t = \alpha\}$ for $k \ge 0$ are i.i.d.\ with $\Q(\tilde \eta^k>0)>0$, we get from the strong law of large numbers that $\sum_{k\ge k_0} \tilde\eta^k = \infty$ a.s.\ But, from \eqref{pfxs} it follows that $\tilde\eta^k \leq \eta^k$ for all $k \ge k_0$, and that yields a contradiction with \eqref{pfxr}. Thus $\P(A) = 0$.
\end{proof}

\begin{proof}[Proof of Proposition \ref{p:cty:tau}]
   The proof is a modification of the last part of the proof of Proposition \ref{p:ODE} that starts below \eqref{pfxw}. Let $(\Omega,\cF,\P)$ be a sufficiently large underlying probability space. 
We reason pointwise for $\P$-a.e.\ $\omega \in \Omega$. On $\bbrack{1,N}$ let the equivalence relation $i \sim j$ be defined by $\liminf_{t \nearrow \tau \wedge T} R^{\{i,j\}}(X_t) = 0$. Take $\bK$ as the family of equivalence classes. By Proposition \ref{p:R:alte} we have that for all $K \in \bK$ and all $j \notin K$
\begin{equation*} 
  \lim_{t \nearrow \tau \wedge T} R^K(X_t) = 0
  \quad \text{and} \quad
  \liminf_{t \nearrow \tau \wedge T} R^{K \cup \{j\}}(X_t) > 0.
\end{equation*}
Hence, for each distinct $K,K' \in \bK$, each $i \in K$ and each $j \in K'$
\begin{equation} \label{pfxn}
  \lim_{t \nearrow \tau \wedge T}  |X_t^i - M^K(X_t)| = 0
  \quad \text{and} \quad\liminf_{t \nearrow \tau \wedge T} |X_t^i - X_t^j| > 0.
\end{equation}
Then, following the computation done on $M(x(t))$ in the proof of Lemma \ref{l:R:sBP}, since a.s.\ for all $K\in \bK$, all $t\in [0,\tau\wedge T)$, 
$$
M^K(X_t) = M^K(X_0) + M^K(B_t) + \frac1{|K|} \sum_{i \in K} \sum_{j \notin K} \int_0^t b^i b^j f(X_s^i - X_s^j) \, ds,
$$
we get that $(M^K(X_t))_{t \in [0, \tau \wedge T)}$ is uniformly continuous because the drift term is bounded on $[0, \tau \wedge T]$. Hence, $\lim_{t \nearrow \tau \wedge T} M^K(X_t)$ exists. Then, by \eqref{pfxn}, a.s.\ $\lim_{t \nearrow \tau \wedge T} X_t^i$ exists in $\R^2$ for all $i \in K$. Since $K \in \bK$ is arbitrary, we obtain that $\lim_{t \nearrow \tau \wedge T} X_t$ exists in $(\R^2)^N$. 

Finally, on the event $\{\tau \leq T\}$, we prove that $X_{\tau-} \in \partial \cE$. Recall $\tau^\e$ from \eqref{taue} and that $\tau^\e \to \tau$ as $\e \to 0$. We take any countable subsequence (not relabelled) along $\e \to 0$. Let $i_\e \neq j_\e$ be accordingly such that $|X_{\tau^\e}^{i_\e} - X_{\tau^\e}^{j_\e}| = \e$, and take a further subsequence of $\e$ (not relabelled) along which $i_\e = i_*$ and $j_\e = j_*$ are fixed. Then, $X_{\tau-}^{i_*} = X_{\tau-}^{j_*}$ and thus $X_{\tau-} \in \partial \cE$. 
\end{proof}

\section{Characterization of collisions}
\label{s:simple-coll}

With having established Proposition \ref{p:cty:tau} on the continuity of $\SC(b)$-processes $(X_t)_{t \ge 0}$ at their lifetime $\tau$, we characterize in this section the type of collisions that can occur at $\tau$. The main result is Theorem \ref{t:3plus}. First,
our terminology of $\SC(b)$-processes is due for an update:

\begin{rem}[Modification of the definition of $\SC(b)$-processes] \label{r:SC}
Henceforth, for any $\SC(b)$-\\
process $(X_t)_{t\ge 0}$ with lifetime $\tau$ (recall the end of Section \ref{s:SC:sol} for this terminology), we employ Proposition \ref{p:cty:tau} to \textit{redefine} $X_t$ as $X_{\tau-}$ for all $t \geq \tau$. This turns $(X_t)_{t \ge 0}$ into a continuous $(\R^2)^N$-valued process. We still refer to $\tau$ as the lifetime, and consider $(X_t)_{t\ge 0}$ to be frozen on $[\tau, \infty)$. Consequently, for any nonempty $K \subset \bbrack{1,N}$, the processes $(M^K(X_t))_{t\ge 0}$ and $(R^K(X_t))_{t\ge 0}$ are continuous on $\R^2$ and $[0,\infty)$ respectively.
\end{rem}

Next, we introduce terminology on the types of collision that can occur at $\tau$. Precisely, we consider collisions as a geometric property of a (static) configuration $x \in (\R^2)^N$ with signs $b \in \{-1,1\}^N$. Recall the cluster separation distance $d^K$ from \eqref{dK}. 

\begin{defn}[Collision types] \label{d:Kcol}
Let $b \in \{-1,1\}^N$, $x \in (\R^2)^N$ and $K \subset \bbrack{1,N}$. We say that there is a \emph{$K$-collision} 
in $x$ if $|K| \geq 2$, $R^K(x) = 0$ and $d^K(x) > 0$. Moreover, we call a $K$-collision with $|K|=2$ and $\sum_{i\in K} b^i = 0$ a \emph{simple collision}. 
Finally, a \emph{same-sign collision} is a $K$-collision with $|\sum_{i \in K} b^i| = |K|$.
\end{defn}

We give some comments to parse this Definition. If there is a $K$-collision in $x$, then no more and no less particles than those in $K$ collide at the same point, which is $M^K(x) \in \R^2$. Note that a configuration $x$ could contain more than one collision. For instance, for three distinct points $y^1, y^2,y^3 \in \R^2$, the configuration $x := (y^1, y^2,y^2,y^1,y^3,y^2) \in (\R^2)^6$ contains two collisions: a $\{1,4\}$-collision at $M^{\{1,4\}}(x) = y^1$ and a $\{2,3,6\}$-collision at $y^2$.
Finally, simple collisions are collisions between precisely two particles of opposite sign.

The main Theorem \ref{t:3plus} states that $X_\tau$ contains precisely one collision, and that this collision is simple. 
To state this precisely, we recall that $\SC(b,x)$-processes are strong Markov processes defined in Definition \ref{d:s-sol} with update in Remark \ref{r:SC} and that $\cE \subset (\R^2)^N$ (see \ref{cE}) is the set of all configurations $x$ without collisions. In addition, we 
introduce (recall Definition \ref{d:Kcol} on the types of collisions)
\begin{align*}
  \cG_2 
  &:= \{ x \in (\R^2)^N : \text{ there is one simple collision and no other collisions in }x \} \\
  &= \{ x \in (\R^2)^N : \exists! \text{ distinct $i,j $ such that } x^i = x^j. \text{ For this pair, } b^i = -b^j = 1 \}.
\end{align*}
Note that $\cG_2 \subset \partial \cE = \cE^c$.

\begin{thm}[Time-distinct, simple collisions only] \label{t:3plus}
Let $\gamma > 0$, $N \geq 2$, $b \in \{-1,1\}^N$, $x \in \cE$, and
    $(X_t)_{t\ge 0}$ be a $\SC(b,x)$-process with lifetime $\tau$.
If $|\sum_{i=1}^N b^i| = N$, then $\P_x(\tau < \infty) = 0$.
Otherwise, $|\sum_{i=1}^N b^i| < N$ and $\P_x(\tau < \infty, \, X_\tau \in \cG_2) = 1$.
\end{thm}

We will show that the single-sign case $|\sum_{i=1}^N b^i| = N$ is simple. In the case $|\sum_{i=1}^N b^i| < N$, the strategy of the proof is roughly as follows. We will show, using the scaling invariance of $SC$-processes (see Proposition \ref{p:invar}\ref{p:invar:scale}), that there exists a `good' set $\cE_N\subset \cE$ and a sequence of stopping times $(\eta^k)_{k\ge 0}$ such that 
\begin{itemize}
    \item whenever $X_{\eta^k} \in \cE_N$, then with uniformly positive probability, $\eta^{k+1} = \tau$ and $X_\tau \in \cG_2$, and in the complementary event, $\eta^{k+1}<\tau$,
    \item whenever $X_{\eta^k}\in \cE \setminus \cE_N$, then a.s.\ 
    $\eta^{k+1}<\tau$ and with uniformly positive probability $X_{\eta^{k+1}} \in \cE_N$.
\end{itemize}
This will imply the result by a Borel-Cantelli like argument.

\begin{proof} 
We rely on the key Propositions \ref{p:no:PPcol},  \ref{p:2coll} and \ref{p:it} that are proven in the subsequent subsections.

In the case $|\sum_{i=1}^N b^i| = N$ it suffices to apply Corollary \ref{c:no:PPcol} for each $K \subset \bbrack{1,N}$ with $|K| \geq 2$. In the remainder we assume $|\sum_{i=1}^N b^i| < N$.

The case $N = 2$ is trivial. Indeed, in this case $\delta < 2$ and, on the event $\{\tau < \infty\}$, we have $X_\tau \in \cG_2 \iff \tau$ is the first $0$-hitting time of $R(X_t)$.
Lemma \ref{l:R:sBP} guarantees that $R(X_t)$ is a $\SB(\delta)$-process until its first $0$-hitting time, which is finite by Theorem \ref{t:BP}\ref{t:BP:less2}.
In the remainder we focus on $N \geq 3$.

\vspace{0.3cm}

We construct iteratively a nondecreasing sequence of finite stopping times $(\eta^k)_{k\ge 0}$ such that $\eta^k = \tau$ for $k$ (random) large enough. In preparation for this, a posteriori, let \comm{[\textit{finite} stopping time means $\eta^k < \infty$] }
\begin{equation}\label{cN}
  c_N := \frac1{7^N N^{\frac32 N}} > 0.  
\end{equation}
Let $p = p(N,\gamma) > 0$ and $q = q(N,\gamma) > 0$ be respectively the deterministic probability values from Propositions \ref{p:2coll} and Proposition \ref{p:it}.
For $y \in (\R^2)^N$, set
\begin{equation*}
  \Dopp(y) := \min_{b^i = -b^j } |y^i - y^j|, \qquad 
    \Dsame(y) := \min_{\substack{ b^i = b^j \\ i \neq j }} |y^i - y^j|.
\end{equation*} 
Note that the minima are attained since $N \geq 3$ and $|\sum_{i=1}^N b^i| < N$ imply that the sets over which the minima are taken are nonempty and finite. 
Finally, we put
\begin{equation*}
  \cE_N := \{ y \in (\R^2)^N : \Dsame(y) \geq c_N \Dopp(y) > 0 \} \subset \cE.
\end{equation*}
We define iteratively
the sequence $(\eta^k)_{k\ge 0}$ with $\eta^0 = 0$ and  for all $k\ge 0$: 
\begin{enumerate}[label=(\roman*)]
  \item On the event $B_k := \{X_{\eta^k} \in \cE_N\} \subset \{\eta^k<\tau \}$, Proposition \ref{p:2coll} below and the strong Markov property provide a stopping time $\eta^{k+1}\in (\eta^k,\tau]$ with $\eta^{k+1} < \infty$ such that \comm{[People understand $\P_Y(f)$ always as that $f$ is a function of $\tilde X$, the MP that starts at $Y$. $f$ includes stuff like stopping times and path-dependent events.]} 
  \comm{\quad [The following should be read as $f(x) := \P_x(\cdots)$, where then Prop \ref{p:2coll} gives a LB on the fct $f$ on $\cE_N$ ]}
  \begin{align} \label{pfxz}
    \indiq_{B_k}\P_x( X_{\eta^{k+1}} \in \cE\cup\cG_2|\cF_{\eta^k}) &=\indiq_{B_k},\\\notag
     \indiq_{B_k} \P_x( X_{\eta^{k+1}} \in \cG_2 |\cF_{\eta^k}) 
     &\geq  \indiq_{B_k}p.
  \end{align}
  \item On the event $A_k := \{X_{\eta^k} \in \cE \setminus \cE_N\} \subset \{\eta^k<\tau\}$, Proposition \ref{p:it} below and the strong Markov property provide a stopping time $\eta^{k+1} \in (\eta^k, \tau)$ such that 
  \begin{equation*} 
 \indiq_{A_k} \P_x(B_{k+1} |\cF_{\eta^k} ) \geq \indiq_{A_k} q.
  \end{equation*}   
\comm{[Proposition \ref{p:it} provides $(\eta_x)_{x \in \cE \setminus \cE_N}$. We take $\eta^{k+1} (\omega) = \eta^{k} (
\omega) + \eta(X_{\eta^k}(\omega),\omega)$. This needs no further comment.]}
  \item On the remaining event $(A_k \cup B_k)^c = \{ X_{\eta^k} \in \partial \cE \} = \{\eta^k = \tau\}$, we set $\eta^{k+1} = \tau$.
\end{enumerate}  
We set $L:=\inf \{ k\ge 0 : \eta^k = \tau\} \in \N_{\ge 1} \cup \{\infty\}$.
It is sufficient to show that $L<\infty$. Indeed, we have
\begin{align}\label{eq:lsubset}
    \{ L< \infty\} \subset \bigcup_{k\ge 1} \{\eta^{k-1}<\tau, \, \eta^k = \tau\} \subset \bigcup_{k\ge 1} B_{k-1} \cap\{ \eta^k = \tau \}.
\end{align}
Then by \eqref{pfxz}, for all $k\ge 1$, $B_{k-1} \subset \{X_{\eta^k} \in \cE \cup \cG_2\} \cup \cN$ for a nullset $\cN$, which implies $B_{k-1} \cap \{ \eta^k = \tau\} \subset \{ X_\tau \in \cG_2\} \cup \cN$, so that together with \eqref{eq:lsubset}, we get $\{L< \infty\} \subset \{X_\tau \in \cG_2\} \cup \cN$.

\vspace{0.3cm}

According to properties (i),(ii) we get according to the strong Markov property,
\begin{align*}
&\mbox{ for all } y \in \cE_N, && \P_y(\eta^2<\tau)
\le \P_y(\eta^1<\tau) 
\le \P_y(X_{\eta^1} \notin \cG_2)
\le 1-p,\\
&\mbox{ for all } y \in \cE \setminus \cE_N, && \P_y(\eta^2=\tau)
= \P_y(\eta^2=\tau, \, B_1)
= \E_y[\indiq_{B_1}\PP_{X_{\eta^1}}( \eta^1=\tau)] 
\ge p \E_y[\indiq_{B_1}]
\ge pq,
\end{align*}
so that for all $y\in \cE$, $\P_y(\eta^2<\tau) \le 1-pq$. Observing that $\{ \eta^{k}<\tau \} = \{ X_{\eta^k} \in \cE \}$, we get using again the strong Markov property that \comm{[Below, first `=': $\{L \ge k+3\} = \{L \ge k+1\} \cap \{ \eta^{k+2} < \tau \}$. Then, use $\indiq$, condition on $X_k$, and use that $\{L \ge k+1\} \in \cF_k$ to pull it out.]}
\begin{align*}
\PP_x(L \ge k+3) &= \E_x[\indiq_{L\ge k+1}\P_{X_{\eta^k}}(\eta^2<\tau)] \le (1-pq)\PP_x(L \ge k+1).
\end{align*}
Since $pq \in (0,1)$, this implies $\PP_x(L<\infty) = 1$.
\end{proof}
The remainder of this section is dedicated to complete the construction of $(\eta^k)_{k\ge 0}$ in the proof above, i.e.\ to prove Propositions \ref{p:no:PPcol}, \ref{p:2coll} and \ref{p:it}.

\subsection{No $K$-collisions with $\delta_K \geq 2$}
\label{s:no:Kcoll}

\begin{prop}[No $K$-collisions with $\delta_K \geq 2$] \label{p:no:PPcol}
Let $\gamma > 0$, $N\ge 2$, $b\in \{-1,1\}^N$, $x\in \cE$, and $(X_t)_{t \ge 0}$ be an $\SC(b,x)$-process. Then, for any $K \subset \bbrack{1,N}$ with $|K| \geq 2$ and $\delta_K \geq 2$ we have
\begin{equation*}
  \PP_x(A_K) = 0, \qquad A_K := \{ \tau < \infty, \text{ there is a $K$-collision in } X_\tau \}.
\end{equation*}
\end{prop}

\begin{proof}
When $|K| = N$, Lemma \ref{l:R:sBP} guarantees that $(R(X_t))_{t \ge 0}$ is on $[0,\tau)$ a $\SB(\delta_K)$-process, which by Theorem \ref{t:BP}\ref{t:BP:geq2} implies $\PP(A_K) = 0$.

For $|K| \leq N-1$, we reason by contradiction. Suppose there exists a $K \subset \bbrack{1,N}$ with $2 \leq |K| \leq N-1$, $\delta_K \geq 2$ and $\PP(A_K) > 0$. Since there are no collisions prior to $\tau$, we have $d^K(X_t) > 0$ for all $t \in [0,\tau)$. On $A_K$, we also have $d^K(X_\tau) > 0$ (recall Definition \ref{d:Kcol}). Then, by the continuity of $(X_t)_{t \ge 0}$, we have
\begin{gather*} 
A_K
= \{ \tau < \infty, \ R^K(X_\tau) = 0, \ \inf_{t\in [0,\tau]} d^K(X_t) > 0 \} 
\subset \bigcup_{T\in \N}\bigcup_{\e \in \Q_{>0}} B_{T,\e}, \\
B_{T,\e} := \Big\{ \tau \leq T, \ R^K(X_\tau) = 0, \ \inf_{t\in [0,\tau]}d^K(X_t) \geq \e \Big\}.
\end{gather*} 
Thus, there exist deterministic constants $T, \e > 0$ such that $\P(B_{T,\e})>0$. With $T,\e$ as above and with $\bK = \{K, K^c\}$, we apply Corollary \ref{c:Girs:bK}. This provides a probability $\Q$ and a $\Q$-$\SB(\delta_K, R^K(x))$-process $(Z_t)_{t\ge 0}$ such that 
\[
  R^K(X_t) = Z_t
  \qquad \text{for all } t \leq \sigma := T \wedge \tau \wedge \inf \{ t \ge 0 : d^K(X_t) < \e \}.
\]
Since $\delta^K \geq 2$, $(Z_t)_{t\ge 0}$ does not hit $0$ a.s. Since $(X_t)_{t\ge 0}$ is continuous, we get $R^K(X_\sigma) = Z_{ \sigma} > 0$. However, on $B_{T,\e}$, we have $\sigma = \tau$ and $R^K(X_\tau) = 0$, which is absurd.
\end{proof}

\begin{cor}[No same-sign collisions] \label{c:no:PPcol}
Proposition \ref{p:no:PPcol} also holds if the condition $\delta_K \geq 2$ is replaced by $|\sum_{i \in K} b^i| = |K|$.
\end{cor}

\begin{proof}
From \eqref{deltaK}, $\gamma > 0$ and $|K| \geq 2$ we obtain 
$
  \delta_K
  > 2(|K|-1)
  \geq 2.
$
We thus conclude using Proposition \ref{p:no:PPcol}.
\end{proof}

\subsection{Uniformly positive probability for simple collisions}

This section is devoted to the following proposition. It  is an improved, quantified version of \cite[Theorem 3.1]{VanMeursPeletierSlangen25}, whose proof construction goes back to at least \cite[Proposition 4]{FournierJourdain17}. In the present version we quantify the lower bound on the probability of a particle collision. 

\begin{prop}\label{p:2coll}   
Let $\gamma > 0$, $N \geq 3$, $b \in \{-1,1\}^N$ be such that $|\sum_{i=1}^N b^i| < N$, and $x \in \cE_N$. Let $( X_t )_{ t \ge 0}$ be an $\SC(b,x)$-process with lifetime $\tau$.  Then, there exists a deterministic constant $p = p(N,\gamma) > 0$ and a stopping time $\eta \leq \tau$ such that $\eta < \infty$ and 
  \begin{equation*}
    \P_x(X_\eta \in \cE \cup \cG_2) = 1, \quad
    \P_x(\eta = \tau, \ X_\eta \in \cG_2) \geq p.
  \end{equation*}
\end{prop}

\begin{rem} \label{r:2coll} 
\begin{enumerate}[label=(\roman*)]
  \item \label{r:2coll:cN} The particular value of $c_N$ (in $\cE_N$) is irrelevant; we state the proof for any constant $c > 0$. In this generalized setting, $p$ depends also on $c$.
  \item \label{r:2coll:N2} In the case $N=2$ a stronger statement holds; see Theorem \ref{t:3plus} and its proof.
  \item \label{r:2coll:eq-event} In the statement of Proposition \ref{p:2coll}, since $X_\eta \in \cE \cup \cG_2$, the events $\{\eta = \tau\}$ and $\{X_\eta \in \cG_2\}$ are equivalent. 
\end{enumerate}
\end{rem}

The proof strategy is to construct a neighborhood $\cH \subset \cE \cup \cG_2$ of $x \in \cE_N$ of `good' configurations in which a selected pair of particles with opposite sign have room to collide while all other particle pairs remain separated. Figure \ref{fig:x1x2} illustrates the setting upon which our construction of $\cH$ is based. We then show that, with positive probability, $\tau$ is smaller than the exit time of $\cH$.

\begin{figure}[h]
\centering
\begin{tikzpicture}[scale=1.8,rotate = 90]
  \def \r {.1}   
   
  \begin{scope}[shift={(0,1)},scale=1] 
    \draw[fill = blue, fill opacity=0.1] (0,0) circle (2); 
  \end{scope}
  
  \begin{scope}[shift={(0,-1)},scale=1] 
    \draw[fill = red, fill opacity=0.1] (0,0) circle (2); 
  \end{scope}
  
  \begin{scope}[shift={(0,1)},scale=1] 
    \draw[fill = white] (0,0) circle (.8); 
    \fill[blue, opacity=0.1] (0,0) circle (.8);
    \fill[red, opacity=0.1] (0,0) circle (.8);
  \end{scope}
  \begin{scope}[shift={(0,-1)},scale=1] 
    \draw[fill = white] (0,0) circle (.8); 
    \fill[blue, opacity=0.1] (0,0) circle (.8);
    \fill[red, opacity=0.1] (0,0) circle (.8);
  \end{scope} 
  
  \begin{scope}[shift={(0,1)},scale=1] 
    \draw (180:2) arc (180:360:2); 
  \end{scope}
  
  \begin{scope}[shift={(0,-1)},scale=1] 
    \draw (0:2) arc (0:180:2); 
  \end{scope}
  
  \draw (0,1-\r) node[right] {$x^1$};
  \begin{scope}[shift={(0,1)}] 
      \filldraw[draw=red, fill=white] (0, 0) circle (\r); 
      \draw[red] (-.7*\r, 0) -- (.7*\r, 0);
      \draw[red] (0, -.7*\r) -- (0, .7*\r);
  \end{scope}
  \draw[fill = black] (0,0) circle (.03);
  \draw[->] (1,0) node[above] {$M^K(x)$} -- (.2,0);
  \draw (0,-1+\r) node[left] {$x^2$};
  \begin{scope}[shift={(0,-1)}] 
      \filldraw[draw=blue, fill=white] (0, 0) circle (\r); 
      \draw[blue] (0, -.7*\r) -- (0, .7*\r);
  \end{scope}
  
  \begin{scope}[shift={(0,-1)},rotate=270] 
    \draw[thin] (\r,0) -- (2,0); 
    \draw (1.2,0) node[above] {$\Dopp(x)$} node[below] {$1$};
  \end{scope} 
  
  
  \begin{scope}[shift={(0,1)},rotate=90] 
    \draw[thin] (\r,0) -- (.8,0) node[midway,above] {\scriptsize $\Dsame(x)$} node[midway,below] {$\geq c$}; 
  \end{scope} 
  
\end{tikzpicture}
\caption{By the definitions of $\Dsame(x)$ and $\Dopp(x)$, positive particles other than $x^1$ are outside of the red regions and negative particles other than $x^2$ are outside of blue ones. Purple regions contain no particles other than $x^1$ and $x^2$.}
\label{fig:x1x2}
\end{figure}
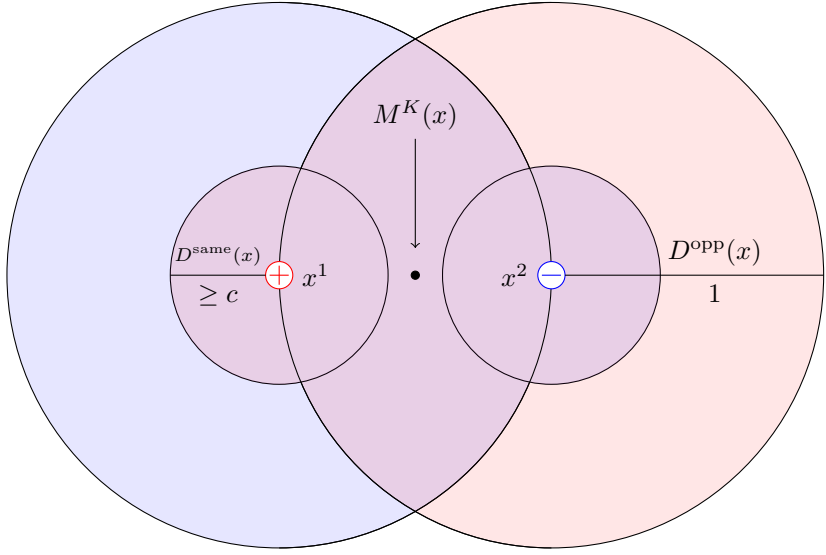 

\begin{proof}[Proof of Proposition \ref{p:2coll}]  
We start with preparations. We write $c$ instead of $c_N$ and only assume that $c > 0$. Since the statement of Proposition \ref{p:2coll} is stronger for smaller $c$, we may put the upper bound $c < \frac12$ without loss of generality. We apply Proposition \ref{p:invar}\ref{p:invar:label} to relabel such that $b^1 = 1$, $b^2 = -1$ and $|x^1 - x^2| = \Dopp(x)$. We take $K = \{1,2\}$. By rescaling $Y_s := \frac1a X_{a^2 s}$ (recall Proposition \ref{p:invar}\ref{p:invar:scale}) with $a := \Dopp(x)$ we may assume that $x$ is such that $\Dopp(x) = 1$. Indeed, since the events in the statement of Proposition \ref{p:2coll} are invariant with respect to the rescaling, $p$ does not depend on $a$.
Since $x \in \cE_N$, $\Dopp(x) = 1$ implies $\Dsame(x) \geq c$. Denoting the rescaled process $(Y_s)_{s\ge 0}$ again by $(X_t)_{t\ge 0}$ (and its collision time by $\tau$), the statement in Proposition \ref{p:2coll} remains unaltered. 
\smallskip

\textbf{Step 1:} Construction and properties of $\cH \subset \cE \cup \cG_2$. Recall the desired description of $\cH$ from the proof strategy. 
The following choice works out convenient for the arguments that follow:
\begin{align*}
  \cH &:= \bigcap_{k=1}^3 \cH_k, \\
  \cH_1 &:= \Big\{ y \in (\R^2)^N : |y^1 - y^2| < 1 + \frac{c^2}6 \Big\},   \\
  \cH_2 &:= \Big\{ y \in (\R^2)^N : |M^K(y) - M^K(x)| < \frac{c^2}{12} \Big\}, \\
  \cH_3 &:= \Big\{ y \in (\R^2)^N : \max_{3 \leq j \leq N} |y^j - x^j| < \frac{c^2}{12} \Big\}.
\end{align*} 
Note that $\cH$ is open. Let $\eta$ be the exit time of $(X_t)_{t \ge 0}$ from $\cH$, capped at $1 \wedge \tau$. When $t\in [0,\eta]$, $\cH_1$ ensures that $|X_t^1 - X_t^2|$ hardly exceeds $|x^1 - x^2| = 1$, $\cH_2$ guarantees that the midpoint $M^K(X_t) = \frac12(X_t^1 + X_t^2)$ remains close to $M^K(x)$, and $\cH_3$ ensures that the other particles stay close to their initial positions.
The precise choices for the bounds in $\cH_k$ are made to quantify the supposed particle separation; see \eqref{pfzg} below. Note that within $\cH$ only one particular kind of collision can happen; a $K$-collision.

\vspace{0.3cm}

Clearly, for all $t\in [0,\eta]$, it follows from the choice of $\cH_3$ and the triangular inequality that 
\begin{equation} \label{pfym}
  \min_{\substack{ i,j \in K^c  \\ i \neq j}} |X_t^i - X_t^j| 
  \geq \min_{\substack{ i,j \in K^c  \\ i \neq j}} \big( |x^i - x^j| - |X_t^i - x^i| - |x^j - X_t^j| \big)
  \geq \Big( c - \frac{c^2}6 \Big) 
  \geq \frac56 c.
\end{equation}
In addition, we claim that
\begin{equation} \label{pfzg}
  \min_{t \in [0,\eta]} d^\bK(X_t) \geq \frac{c^2}{12},
\end{equation}
where $\bK = \{K\} \cup \{\{i\}, i\in K^c\}$ (recall \eqref{dbK}). Given this claim (proved below), we obtain from \eqref{pfym} that
\begin{equation*} 
  \min \big\{ | X_t^i - X_t^j | : t \in [0,\eta], \ i \neq j, \ \{i,j\} \neq K \big\} \geq \frac{c^2}{12}.
\end{equation*}
This proves $P(X_\eta \in \cE \cup \cG_2) = 1$, which is the first part of Proposition \ref{p:2coll}.

Next we prove the claim in \eqref{pfzg}. Thanks to \eqref{pfym} it is left to show that $|X_t^i - X_t^j| \geq \frac1{12} c^2$ for all $t \in [0,\eta]$, each $i \in K$ and each $j \in K^c$.
Let
\begin{equation*}
  D_t^K
  := \min_{3 \leq j \leq N} \big| X_t^j - M^K(X_t) \big|.
\end{equation*}
First, we show that
\begin{equation} \label{pfxe} 
  D_0^K  
  \geq \frac12 \sqrt{1 + 2c^2}.
\end{equation}
To prove \eqref{pfxe}, we choose the axes such that $x^1 = (0,0)$ and $x^2 = (1,0)$. Then, if $j\ge 3$, (we can suppose by symmetry that $b^j = -b^1= b^2$), we have that $|x^j|\ge \Dopp(x) = 1$ and $|x^j-(1,0)|\ge \Dsame(x) = c$. Writing $x^j = (x_1^j,x_2^j)$ we compute
\begin{align*}
    |x^j - M^K(x)|^2 
    &= \Big|x^j - \Big(\frac12,0\Big)\Big|^2 
    = (x_1^j)^2 - x_1^j + \frac14 + (x_2^j)^2 \\
    &= \frac{|x^j|^2}{2} + \frac{|x^j-(1,0)|^2}{2} - \frac{1}{4}
    \ge \frac14 +\frac{c^2}{2}.
\end{align*}
This completes the proof of \eqref{pfxe}.

\vspace{0.3cm}

Using $\sqrt{1+x} \geq 1 + \frac x3$ for $0 \leq x \leq 3$, we get
$
  D_0^K  
  \geq \frac12 + \frac13 c^2
$.

Let $t \in [0, \eta]$ and $j \in \bbrack{3,N}$. Since $X_t \in \cH_2 \cap \cH_3$, we have 
\begin{equation*} 
  |X_t^j - M^K(X_t)|
  \geq |x^j - M^K(x)| - |X_t^j - x^j| - |M^K(X_t) - M^K(x)|
  \geq D_0^K - \frac{c^2}{12} - \frac{c^2}{12} \geq \frac12 + \frac{c^2}6.
\end{equation*}
Using that $X_t \in \cH_1$, we further obtain
\begin{align*} 
  |X_t^j - X_t^1|
  &\geq |X_t^j - M^K(X_t)| - \frac12 |X_t^1 - X_t^2|  
  \geq \frac12 + \frac{c^2}6 - \frac12 \Big(1 + \frac{c^2}6 \Big) = \frac{c^2}{12}.
\end{align*}
A similar estimate holds for $X_t^2$ by symmetry. This proves the claim in \eqref{pfzg}.
\smallskip

\textbf{Step 2:} According to Corollary \ref{c:Girs:bK} applied with $\bK$ defined in Step 1, $\e = \frac1{12} c^2$ and $T=1$, there exist a probability $\Q$ equivalent to $\P$,  a constant $C = C(N,\gamma,c) > 0$ such that $C^{-1} \le \E_{\P}[(\dd \Q/\dd \P)^2] \le C$, and $\Q$-independent processes $(S_t^{K})_{t \ge 0}$, $((S^j_t)_{t\ge 0})_{j\in \ig 3,N\id}$ ($\Q$-Brownian motions) and $(Z_t^{K})_{t \ge 0}$ (a $\Q$-$\SB(\delta_{K}, R^{K}(x))$-process) (see Corollary \ref{c:Girs:bK} for the precise properties) such that 
\begin{equation} \label{pfxi} 
    M^{K}(X_t) = S_t^{K}, \quad 
    R^{K}(X_t) = Z_t^{K}
    \quad \mbox{ and }\quad X^j_t = S^j_t \quad \text{for all $j\in \ig 3,N\id$ and all } t \in [0, \sigma],
\end{equation}
where $\sigma := 1 \wedge \tau \wedge \inf \{ t \ge 0 : d^\bK(X_t) < \frac1{12} c^2 \}$. By \eqref{pfzg} we have $\sigma \geq \eta$.

\vspace{0.3cm}

Recall from Remark \ref{r:2coll}\ref{r:2coll:eq-event} that it is left to show that $\P(\eta = \tau) \geq p$. Since $C^{-1}\le \E_\P[(\dd \Q/\dd \P)^2]$ $\le C$, we get by the Cauchy-Schwartz inequality  \comm{[p.70.7 for details]} 
\begin{equation} \label{pfxl}
  \Q(\eta = \tau)^2 \leq \P(\eta = \tau) \E_{\P}[(\dd \Q/\dd \P)^2] \leq C \P(\eta = \tau),  
\end{equation}
and thus it is sufficient to show that $\Q(\eta = \tau) \geq q_0$ for some deterministic constant $q_0 = q_0(N,c,\gamma) > 0$. From the definition of $\eta$ we obtain that
\[
  \{\eta = \tau\}
  \supset \{ X_\eta \in \cH, \ \tau \leq 1 \}.
\]
Then, from \eqref{pfxi} we obtain that $\{ X_\eta \in \cH, \ \tau \leq 1 \} \supset \cap_{k=1}^3 A_k$, where
\begin{align*}
  A_1 &:= \Big\{ \min_{t\in [0,1]} Z_t^K = 0, \quad 
                     \max_{t\in [0,1]} Z_t^K < \frac12 \Big( 1 + \frac{c^2}6 \Big)^2 \Big\},\\
  A_2 &:= \Big\{ \sup_{t\in [0,1]}|S_t^K - M^K(x)| < \frac{c^2}{12} \Big\}, \\
  A_3 &:= \Big\{ \sup_{t\in [0,1]}\max_{3 \leq j \leq N} |S_t^j -  x^j| < \frac{c^2}{12} \Big\}.
\end{align*}
Moreover, by the $\Q$-independence of the processes in \eqref{pfxi}, we obtain that the events $\{A_k\}_{k=1}^3$ are independent, and thus
\[
  \Q(\eta = \tau)
  \geq \Q( \cap_{k=1}^3 A_k )
  = \prod_{k=1}^3 \Q(A_k).
\]
Clearly, there exist deterministic constants $q_2 = q_2(c)>0$ and $q_3 = q_3(N,c)>0$ such that $\Q(A_2) \ge q_2$ and $\Q(A_3) \ge q_3$. Finally, noting that $\delta^K = 2(1-\gamma) < 2$ (recall \eqref{deltaK}) and $R^K(x) = \frac12 |x^1 - x^2|^2 = \frac12$, Theorem \ref{t:BP}\ref{t:BP:less2} provides a deterministic constant $q_1 = q_1(c,\gamma) >0$ such that $\Q (A_1) \ge q_1$. This completes the proof of the inequality in Proposition \ref{p:2coll}.
\end{proof}

\subsection{Entering $\cE_N$ happens with uniformly positive probability}
\label{s:it}

This section is devoted to the following proposition:

\begin{prop} \label{p:it}
Let $\gamma > 0$, $N \geq 3$, $b \in \{-1,1\}^N$, $x \in \cE \setminus \cE_N$, and $( X_t )_{t \ge 0}$ be a $\SC(b,x)$-process with lifetime $\tau$ and associated filtered probability space $(\Omega,\cF,(\cF_t)_{t\ge 0},\P)$. Then, there exist $q = q(\gamma,N) > 0$ and a stopping time $\eta$ such that $\eta \in (0,\tau)$ a.s.\ and \comm{[ we put $(\Omega,\cF,(\cF_t)_{t\ge 0},\P)$ in the statement whenever convenient for the proof or its use]}
\begin{equation*}
  \P_x(X_\eta \in \cE_N) \geq q.
\end{equation*}
\end{prop}

The proof is involved. We first describe the complexity of Proposition \ref{p:it} and the proof strategy.

Proposition \ref{p:it} is intuitively obvious. Indeed, $X_t \in \cE \setminus \cE_N$ is equivalent to $0 < \Dsame(X_t) < c_N \Dopp(X_t)$, which means that there are at least two particles of the same sign very close together with respect to the nearest particle of opposite sign. This is an unlikely configuration to appear in the $\SC$-process $(X_t)_{t \ge 0}$, because both the Brownian motion and the particle interaction tend to separate particles of equal sign.

The main difficulty with making this rigorous is that we have no lower bound on $\Dsame(X_t)$ and thus no control on how close $X_t$ is to the singularity of the drift. Therefore, we will work with $(M^K(X_t))_{t \ge 0}$ and $(R^K(X_t))_{t \ge 0}$. We will start by dividing $x$ into clusters as $\{ x^K \}_{K \in \bK}$ (recall \eqref{xK}) such that each $x^K$ contains particles of the same sign that are very close together and separated from the other particles (see Figure \ref{fig:clusters}). Then, by Corollary \ref{c:Girs:bK}, $(M^K(X_t))_{t \ge 0}$ behaves like a Brownian motion, and for each cluster with two or more particles, $(R^K(X_t))_{t \ge 0}$ behaves like a $\SB(r^K,\delta_K)$ process with $r^K > 0$ small and $\delta_K > 2$. It is then easy to show that with positive probability, until some stopping time $\zeta < \tau$, $(M^K(X_t))_{t \ge 0}$ does not move much away from $M^K(x)$ and $(R^K(X_t))_{t \ge 0}$ reaches a value much larger than $r^K$. 

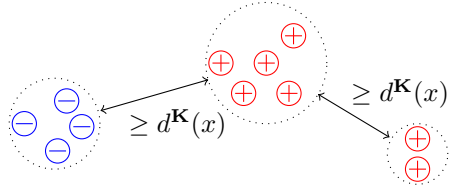
\begin{figure}[h]
\centering
\begin{tikzpicture}[scale=.2]
    \def \r {.8}
    \def \lgray {black!20!white}
    
    \begin{scope}[shift={(14,4)}]
    \foreach \x/\y in {0/0, 1.8/1.8, -3/0, -1.5/-2, 1.5/-2}{
    \begin{scope}[shift={(\x,\y)}]
        \filldraw[draw=red, fill=white] (0,0) circle (\r); 
        \draw[red] (-.7*\r,0) -- (.7*\r,0);
        \draw[red] (0,-.7*\r) -- (0,.7*\r);
    \end{scope}
    }
    \draw[dotted] (0,0) circle (4);
    \begin{scope}[rotate = 196]
      \draw[<->] (4,0) -- (11.4,0); 
    \end{scope}
    \begin{scope}[rotate = 329]
      \draw[<->] (4,0) -- (9.5,0);
    \end{scope}

    \draw (-2,-4) node[left] {$\geq d^\bK(x)$};
    \draw (5,-2) node[right] {$\geq d^\bK(x)$};
    \end{scope}

    \begin{scope}[shift={(24,-2)}]
    \foreach \x/\y in {0/-1, 0/1}{
      \begin{scope}[shift={(\x,\y)}] 
      \filldraw[draw=red, fill=white] (0, 0) circle (\r); 
      \draw[red] (-.7*\r, 0) -- (.7*\r, 0);
      \draw[red] (0, -.7*\r) -- (0, .7*\r);
      \end{scope}
    }
    \draw[dotted] (0,0) circle (2);
    \end{scope}
    
    \begin{scope}[shift={(0,0)}]
    \foreach \x/\y in {-2/0, .8/1.5, 1.8/-.2, .2/-1.8}{
      \begin{scope}[shift={(\x,\y)}] 
        \filldraw[draw=blue, fill=white] (0, 0) circle (\r); 
        \draw[blue] (-.7*\r, 0) -- (.7*\r, 0);
      \end{scope}
    }
    \draw[dotted] (0,0) circle (3);
    \end{scope}               
\end{tikzpicture}
\caption{Example of a configuration $x$ and a partition $\bK$ that separates the positive from the negative particles.}
\label{fig:clusters}
\end{figure}

The problem is that this does not give the information we seek, namely that $\Dsame(X_\zeta) \geq c_N \Dopp(X_\zeta)$. Indeed, if $|K| \geq 3$, then $R^K(X_\zeta) \gg r^K$ gives no lower bound on $\Dsame(X_\zeta^K)$; it only says that not \textit{all} particles in the cluster $X_\zeta^K$ are close. 

The proof strategy of Proposition \ref{p:it} is to solve this problem by an iterative argument. If $R^K(X_\zeta) \gg r^K$, then we show that we can split the cluster $X_\zeta^K$ into at least two sub-clusters $X_\zeta^{K_1}$ and $X_\zeta^{K_2}$ that are sufficiently separated from each other (Proposition \ref{p:cluster:split}). Then, similar to the division of $x$ into $\{ x^K \}_{K \in \bK}$, but with different division parameters (i.e.\ the permitted values for $d^\bK(x)$ and the diameters of the clusters, see Figure \ref{fig:clusters}), we can divide $X_\zeta$ into $\{ X_\zeta^K \}_{K \in \bL}$ for a partition $\bL$ which is strictly finer than $\bK$ (Lemma \ref{l:it:step}). Iterating over such partitions and the corresponding stopping times, we arrive in at most $N-2$ steps with positive probability  at some stopping time $\eta$ at the trivial partition consisting of $N$ singletons, i.e.\ all particles $X_\eta^i$ are separated by some distance $d > 0$. Then, in particular, $\Dsame(X_\eta) \geq d$, which gives sufficient control to conclude. This iterative argument results in the exponential decay of $c_N$ in $N$ (recall \eqref{cN}).
\medskip

The proof of Proposition \ref{p:it} is given at the end of this section. First, we introduce notation and prepare several lemmas. 

Given $N \geq 2$, $b \in \{-1,1\}^N$, a partition $\bK$ of $\bbrack{1,N}$ and a configuration $x \in (\R^2)^N$, we recall the minimal cluster separation distance $d^\bK(x)$ from \eqref{dbK} and that $|\bK|$ is the number of elements of $\bK$. In addition, we define:
  \begin{itemize}
     \item a \textbf{refinement} $\bL$ of $\bK$ as a partition of $\bbrack{1,N}$ such that for all $L \in \bL$ there exists a $K \in \bK$ such that $L \subset K$. It is \textbf{strict} if $|\bL| \geq |\bK| + 1$ (or, equivalently, $\bL \neq \bK$);
     \item `$\bK$ separates the positive from the negative particles' as the condition that for any $K \in \bK$ and any $i,j \in K$ we have $b^i = b^j$. Note that this condition is independent of $x$; \comm{[Before I thought that this can be seen from comparing $\Dopp(x)$ and $d^\bK(x)$, but that is false. Indeed, for $b = (1,1,-1)$, $x^3$ far away, and $\bK = \{1, (2,3)\}$, we have $\Dopp(x) > d^\bK(x)$ but no pos-neg separation. Such connections come when working with associated partitions defined below]}
   \end{itemize}  

\begin{defn}[Associated partition] \label{d:bK} 
    Let $N \geq 2$, $x \in (\R^2)^N$ and $d_0 > 0$. The associated partition $\bK(x,d_0)$ is defined as the set of connected components of the graph $G = (V, E)$, where $V = \bbrack{1,N}$ is the vertex set and 
\[
  E = \Big\{ (i,j) : i,j \in V \text{ distinct, } |x^i - x^j| < d_0 \Big\}
\]
is the edge set.
\end{defn}

An alternative interpretation of the associated partition is as follows. Given a configuration $x$, there can be various partitions $\bK$ with the same value $d_0$ for $d^\bK(x)$. The finest partition among these is the associated partition. This is made precise by the stronger statement \ref{l:bK:e} in the 
following lemma, which list several properties of associated partitions. 

\begin{lem}[Properties of associated partitions] \label{l:bK}
    Let $N \geq 2$, $x \in (\R^2)^N$ and $d_0 > 0$ and $\bK_0 := \bK(x,d_0)$ be the associated partition. Then:
    \begin{enumerate}[label=(\alph*)]
        \item \label{l:bK:a} $d^{\bK_0}(x) \geq d_0$,
        \item \label{l:bK:b} if $0 < d_1 \le d_0$, then $\bK(x,d_1)$ is a refinement of $\bK(x,d_0)$,
        \item \label{l:bK:e} for any partition $\bL$ of $\bbrack{1,N}$, $\bK(x, d^\bL(x))$ is a refinement of $\bL$,
        \item \label{l:bK:d} for any $K \in \bK(x,d_0)$, $R^K(x) < \frac12 |K|^3 d_0^2$.
    \end{enumerate}
\end{lem}

\begin{proof}
\ref{l:bK:a} and \ref{l:bK:b} are obvious. To prove \ref{l:bK:e}, set $\bK := \bK(x, d^\bL(x))$. Let $K \in \bK$ and $i \in K$ be arbitrary. Take $L \in \bL$ such that $i \in L$. It is sufficient to show that $j \in K \implies j \in L$. This is obvious if $j = i$. If $j \neq i$, then there exists a path $\{i_k\}_{k=0}^m$ on the graph $G$ from Definition \ref{d:bK} with $i_0 = i$, $i_m = j$ and $|x_{i_k} - x_{i_{k-1}}| < d^\bL(x)$ for all $k \in \bbrack{1,m}$. By the definition of $d^\bL(x)$, since $i = i_0 \in L$ and $|x_{i_1} - x_{i_0}| < d^\bL(x)$, we have $i_1 \in L$. Iterating this over the path, we obtain that $j = i_m \in L$.

We prove \ref{l:bK:d} by a similar argument. For any $i,j \in K$, take the shortest path $\{i_k\}_{k=0}^m$ among those mentioned above. Then, $|x^i - x^j| < |K|d_0$ and thus
\[
  R^K(x)
  = \frac1{2 |K|} \sum_{i,j=1}^{|K|} |x^i - x^j|^2
  < \frac{|K|^3}{2} d_0^2.
\]
\end{proof}

When splitting a cluster $x^K$ 
into two subclusters $x^{K_1}$ and $x^{K_2}$, we require a sufficient lower bound on $d^{K_i}(x)$ for $i=1,2$. The following proposition provides certain $K_i$ for which a convenient lower bound holds in terms of $|K|$ and $R^K(x)$. It is stated for $K = \bbrack{1,N}$. Recall that $R = R^{\ig 1,N\id}$ and that $d^L = d^{L^c}$.

\begin{prop}[Splitting a cluster] \label{p:cluster:split}
Let $N \geq 2$. For all $x\in (\R^2)^N$ there exists $K \subset \bbrack{1,N}$ such that 
\begin{equation*}
  1 \leq |K| \leq N-1 \quad \text{and} \quad d^K(x) \geq \sqrt{ \frac{2 R(x)}{N^3} }.
\end{equation*}
\end{prop}

\begin{proof}
If $R(x) = 0$, then the inequality is trivial. When $R(x) > 0$, 
we reason by contradiction. Suppose no such $K$ exists. Then this implies that $\bK(x,\alpha) = \{ \bbrack{1,\dots,N} \}$, where $\alpha := \sqrt{2 R(x) N^{-3}}$. Applying Lemma \ref{l:bK}-(d), we get that 
$$
R(x)< \frac12 N^3 \alpha^2 = R(x), 
$$
which is absurd.
\end{proof} 
 
We use the factor $\sqrt{2 N^{-3}}$ to quantify parts of the following lemma. Ideally, given an $\SC(b,x)$-process $(X_t)_{t \ge 0}$ with lifetime $\tau$, we want to find a strict refinement $\bL$ of $\bK$ such that $d^\bL(x) \geq \sqrt{2 N^{-3}} d^\bK(x)$. This is, however, not possible if all particles within each cluster $x^K$ are very close (i.e.\ $R^K(x)$ is small) with respect to $d^\bK(x)$. In this case, we  show that $(R^K(X_t))_{t \ge 0}$ behaves for each $K \in \bK$ as an independent $\SB(b^K,R^K(x))$-process until particles from other clusters get too close. If $R^K(X_\zeta)$ gets large enough for some $K$ at some stopping time $\zeta$, then we apply Proposition \ref{p:cluster:split} to split $K$ into two subsets. We construct $\bL$ based on this splitting, and set the parameters to obtain the desired bound $d^\bL(X_\zeta) \geq \sqrt{2 N^{-3}} d^\bK(X_\zeta)$. Our task is to show that this construction of $\zeta$ and $\bL$ succeeds with uniformly positive probability. 

\begin{lem}[A single refinement step] \label{l:it:step}
Let $\gamma > 0$, $N \geq 3$, $b \in \{-1,1\}^N$, $x \in \cE \setminus \cE_N$, $\bK$ be a partition of $\bbrack{1,N}$ such that $d^\bK(x) > \Dsame(x)$ and $\bK$ separates the positive from the negative particles.
Let $(\Omega,\cF,(\cF_t)_{t\ge0},\P)$ be a filtered probability space on which is defined a $\SC(b,x)$-process $( X_t )_{t \ge 0}$ with lifetime $\tau$. Then, there exists a stopping time $\zeta$ and an $\cF_\zeta$-measurable partition $\bL$ such that, a.s.
\begin{enumerate}[label=(\alph*)]
  \item \label{pp4} $\zeta < \tau$ and $\zeta \leq d^\bK(x)^2$,
  \item \label{pp7} $\bL$ is a strict refinement of $\bK$,
  \item \label{pp6} if $\zeta = 0$, then $d^\bL(x) \geq \frac17 \sqrt{  2N^{-3} } d^\bK(x)$.
\end{enumerate}  
Furthermore, there exists a deterministic constant $q_* = q_*(\gamma,N) > 0$ such that with probability greater than $q_*$ the intersection of the following events hold:
\begin{enumerate}[resume,label=(\alph*)]
  \item \label{pp1} $
  d^\bL(X_\zeta) \geq \frac17 \sqrt{  2N^{-3} } d^\bK(x)$,
  \item \label{pp3} $
  \Dopp(X_\zeta) \leq \frac67 d^\bK(x) + \Dopp(x)$.
\end{enumerate}
\end{lem} 

\begin{proof}
We start with some preparations. Since $\bK$ separates the positive from the negative particles, we have $|\bK| \geq 2$. From $d^\bK(x) > \Dsame(x)$ it follows that there exists a $K \in \bK$ with $|K| \geq 2$. In particular, $|\bK| \leq N-1$. In addition, we may assume that $d^\bK(x) = 1$. Indeed, applying Proposition \ref{p:invar}\ref{p:invar:scale} with $a := d^\bK(x)$, we 
 obtain that $(\tilde X_t)_{t \ge 0}$ defined by $\tilde X_t := \frac1a X_{a^2t}$ is an $\SC(b, \frac xa)$-process with lifetime $\tilde \tau := a^{-2} \tau$, adapted to $(\tilde \cF_t)_{t\ge 0} := (\cF_{a^2t})_{t\ge 0}$. Since the events in Lemma \ref{l:it:step}\ref{pp1},\ref{pp3} are invariant under rescaling by $a$, the constant $q_*$ does not depend on $a$.

Let $\bK_1$ be the partition associated to $(x, d^\bK(x))$ (recall Definition \ref{d:bK}). Take 
\begin{equation} \label{pfxc}
    d_N 
    := \frac17 \sqrt{ \frac2{N^3} }
    < 1 = d^\bK(x)
\end{equation} 
and let $\bK_N$ be the partition associated to $(x, d_N)$. We observe that:
\begin{itemize}
    \item by Lemma \ref{l:bK}\ref{l:bK:e}, $\bK_1$ is a refinement of $\bK$, 
    \item by Lemma \ref{l:bK}\ref{l:bK:b}, $\bK_N$ is a refinement of $\bK_1$, and hence $d^{\bK_N}(x) \leq d^\bK(x)$,
    \item by Lemma \ref{l:bK}\ref{l:bK:a}, $
  d^{\bK_N}(x) 
  \geq d_N
  = \frac17\sqrt{2 N^{-3}}
$,
and
  \item by Lemma \ref{l:bK}\ref{l:bK:d}, 
\end{itemize}
\begin{equation} \label{pfyh}
  R^K(x)
  < \frac{|K|^3}2d_N^2
  < \frac1{49}
  \qquad \text{for all } K \in \bK_N.
\end{equation}
From the first two observations, we either have that $\bK_N$ is a strict refinement of $\bK$ or that $\bK_N = \bK$. In the former case, we take $\zeta := 0$ and $\bL := \bK_N$. Then, it is easy to see that Lemma \ref{l:it:step}\ref{pp4}--\ref{pp6} hold.
\medskip

In the remainder we assume the latter case, i.e.\ $\bK_N = \bK$. We start by constructing $\zeta$. For each $K \in \bK$ and all $t \ge 0$, we denote $R_t^K := R^K(X_t)$ and $M_t^K := M^K(X_t)$. By \eqref{pfyh},
\begin{equation} \label{pfye}
  \eta_K := \inf \Big\{ t\ge 0 : R_t^K \ge \frac1{49} \Big\}, \qquad 
  \eta := \min_{K \in \bK} \eta_K > 0. 
\end{equation}
In addition, we consider the stopping time 
\begin{equation} \label{pfyg}
  \xi := \inf \Big\{ t \ge 0 : \max_{K \in \bK} |M_t^K - M_0^K| \geq \frac17 \Big\} > 0.
\end{equation}
Finally, we set 
\begin{equation*} 
  \zeta := 1 \wedge \tau \wedge \eta \wedge \xi > 0 \quad \mbox{a.s.}
\end{equation*}
This already proves \ref{pp6} and the second upper bound in \ref{pp4}.

Next we prove the auxiliary estimates 
\begin{align} \label{pfxb}
    \max_{0 \leq t \leq \zeta} \max_{K \in \bK} \max_{i \in K} |X_t^i - M_t^K| &\leq \frac17,
    \\\label{pfyb}
    \min_{0 \leq t \leq \zeta} d^\bK(X_t) &> \frac17. 
\end{align}
Take $t \in [0,\zeta]$, $K \in \bK$ and $i \in K$. From the definition of $R^K$ we obtain $|X_t^i - M_t^K|^2 \leq  R_t^K$. Then, \eqref{pfxb} follows from $t \leq \zeta \leq \eta$ and \eqref{pfye}. To prove \eqref{pfyb}, take $t \in [0,\zeta]$, distinct $K, K' \in \bK$, $i \in K$ and $j \in K'$ such that $d^\bK(X_t) = |X_t^i - X_t^j|$. Then, using \eqref{pfxb},
      \begin{equation*} 
        d^\bK(X_t)
        \geq |M_t^K - M_t^{K'}| - |X_t^i - M_t^K| - |X_t^j - M_t^{K'}|
        \geq |M_t^K - M_t^{K'}| - \frac27.
      \end{equation*}
By $t \leq \zeta \leq \xi$ and \eqref{pfyg}, we have $|M_t^K - M_t^{K'}| \geq |M^K(x) - M^{K'}(x)| - \frac27$. Applying \eqref{pfxb} at initial time, we obtain
\[
  |M^K(x) - M^{K'}(x)|
  \geq |x^i - x^j| - |M^K(x) - x^i|  - |x^j - M^{K'}(x)|
  \geq |x^i - x^j| - \frac27.
\]
Clearly, $|x^i - x^j| \geq d^\bK(x) = 1$. Trailing back the inequalities, we obtain \eqref{pfyb}.
\vspace{0.3cm}
      
Next we complete the proof of \ref{pp4} by showing that $\zeta < \tau$. Take any $t \in [0,\zeta]$ and any $i,j$ with $b^i = -b^j$. Since $\bK$ separates the positive and negative particles, $i$ and $j$ belong to different index sets of $\bK$. Hence, by \eqref{pfyb}
      \begin{equation*} 
        |X_t^i - X_t^j|
        \geq d^\bK(X_t)
        \geq \frac17. 
      \end{equation*}
Thus, $\Dopp(X_t) \geq \frac17$. Consequently, during $[0,\zeta]$ the only type of collision that can happen is a same-sign collision. Corollary \ref{c:no:PPcol} states that this does not happen. Hence, $\zeta < \tau$.

\vspace{0.3cm}

In preparation for constructing $\bL$, we prove
\begin{equation} \label{pfyd}
  \P(A) \geq q_*,
  \qquad A := \{ \zeta = \eta < 1 \wedge \tau \}
\end{equation}
for some $q_* = q_*(\gamma, N) > 0$.
According to Corollary \ref{c:Girs:bK} applied with $\e = \frac18$ and $T=1$, there exist a probability $\Q$ equivalent to $\P$,  a constant $C = C(N,\gamma) > 0$ such that $C^{-1} \le \E_{\P}[(\dd \Q/\dd \P)^2] \le C$, and $\Q$-independent processes $(S_t^K)_{t \ge 0}$ ($\Q$-Brownian motions) and $(Z_t^K)_{t \ge 0}$ ($\Q$-$\SB(\delta_K, R_0^K$-processes) (see Corollary \ref{c:Girs:bK} for the precise properties) such that 
\begin{equation} \label{pfxq}
    M_t^K = S_t^K
    \quad \text{and} \quad 
    R_t^K  = Z_t^K
    \quad \text{for all $K \in \bK$ and all } t \in [0, 1 \wedge \tau \wedge \sigma],
\end{equation} 
where $\sigma := \inf \{ t \ge 0 : d^\bK(X_t) < \frac18 \}$. From \eqref{pfyb} we obtain $\sigma > \zeta$. In addition, by applying the Cauchy-Schwartz inequality (see \eqref{pfxl} for details), we obtain $\P(A) \geq \frac1C \Q(A)^2$.

In the remainder of the proof of \eqref{pfyd}, we bound $\Q(A)$ from below.
We do this by introducing  
$\Q$-independent events $A_1, A_2 \in \cF$ with $\Q(A_1), \Q(A_2) > 0$ and $A_1 \cap A_2 \subset A$.
First, given an integer $k \geq 1$ and a $2k$-dimensional $\Q$-Brownian motion  $(W_t^{2k})_{t \ge 0}$ starting from $0$ and scaled by $k^{-1/2}$, let 
\begin{equation*} 
  q_1(k) := \Q \Big( \sup_{0 \leq t \leq 1} | W_t^{2k} | < \frac17 \Big) > 0.
\end{equation*}
Set $q_2:= \min_{k\in \ig 1,N\id} q_1(k) >0$. In particular, for all $K \in \bK$,
 since $(S_t^K)_{t \ge 0}$ is a $2|K|$-dimensional $\Q$-Brownian motion scaled by $|K|^{-1/2}$ and started at $M^K(x)$,
\begin{equation*} 
  \Q \Big( \sup_{0 \leq t \leq 1} | S_t^K -  M^K(x) | < \frac17 \Big) \geq q_2.
\end{equation*}
Second, let $K \in \bK$ be such that $|K| \geq 2$. Since $|\bK| \leq N-1$, such a $K$ exists. Since $\bK$ separates the positive from the negative particles, we have by \eqref{deltaK} that $\delta_K = \gamma |K|(|K|-1) + 2(|K|-1) > 2$. Hence, since $(Z_t^K)_{t \ge 0}$ is a $\Q$-$\SB(\delta_K, R^K(x))$-process, we have by Theorem \ref{t:BP}\ref{t:BP:geq2} that $(Z_t^K)_{t \ge 0}$ does not hit $0$ a.s. Thus, $(Z_t^K)_{t \ge 0}$ satisfies the SDE \eqref{R:Bessel:SDE} for all $t \ge 0$. Then, by the comparison theorem (see e.g.\ \cite[Theorem 3.7, p394]{RevuzYor99}) there exists a $SqB(2,0)$-process $(Z_t)_{t \ge 0}$ (not stopped at $0$) under $\Q$ such that a.s.\ $Z_t^K \geq Z_t$ for all $t \ge 0$. Finally, from Remark \ref{r:BP} we obtain
\begin{equation*}
  \Q \Big( \sup_{0 \leq t \leq 1}  Z_t^K > \frac1{49} \Big) \geq \Q \Big( \sup_{0 \leq t \leq 1}  Z_t > \frac1{49} \Big) =: q_3 > 0.
\end{equation*}
Combining the events above and applying the independence of $(S_t^K)_{t\ge 0}$ and $(Z_t^K)_{t\ge 0}$, we obtain $\Q(A_1 \cap A_2) \geq q_2^{N-1} q_3$, where
\begin{align*}
  A_1 &:= \Big\{ \max_{K \in \bK} \sup_{0 \leq t\le 1} |S_t^{K} - M^{K}(x)| < \frac17 \Big\}, \qquad 
  A_2 := \Big\{ \sup_{0 \leq t \leq 1} \max_{K \in \bK} Z_t^K > \frac1{49}
       \Big\}.
\end{align*} 

Finally, we show that $\Q(A_1 \cap A_2) \leq \Q(A)$, which is the last ingredient for concluding \eqref{pfyd} with $q_* = \frac1C q_2^{2(N-1)} q_3^2$. Since $\zeta < \tau$, we have $\zeta = 1 \wedge \eta \wedge \xi < \tau$. On $A_2$, we have $\eta < 1$ and thus $\zeta = \eta \wedge \xi < 1 \wedge \tau$.  Recalling \eqref{pfxq}, we have on $A_1$ that $\xi \geq 1 \wedge \tau \wedge \sigma$. Since $\sigma > \zeta$, we have on $A_1$ that 
$\sigma \geq 1 \wedge \tau \implies \xi \geq 1 \wedge \tau$ and $\sigma < 1 \wedge \tau \implies \xi \geq \sigma > \zeta$. Hence, in either case, on $A_1 \cap A_2$, we have $\xi > \zeta$. Thus, $A_1 \cap A_2 \subset A$. 
\vspace{0.3cm}

With \eqref{pfyd} established, we construct $\bL$ separately on $A^c$ and $A$, and prove the remaining statements \ref{pp7}, \ref{pp1} and \ref{pp3}. 
On $A^c$, we take $\bL$ as the finest partition (i.e.\ $|\bL| = N$). Then, \ref{pp7} obviously holds on $A^c$. By \eqref{pfyd} it is sufficient to prove \ref{pp1} and \ref{pp3} only on $A$.
  
On $A$, let $\bL$ be the partition associated with $(X_\zeta, d_N)$, where we recall $d_N = \frac17 \sqrt{  2 N^{-3} } $ from \eqref{pfxc}. Obviously, $\bL$ is $\cF_\zeta$-measurable, and \ref{pp1} follows directly from Lemma \ref{l:bK}\ref{l:bK:a} (recall $d^\bK(x) = 1$).
To prove \ref{pp7}, we first apply \eqref{pfyb} with $t = \zeta$ to obtain 
\begin{equation*} 
    d^{\bK}(X_{\zeta}) 
    \geq \frac17 
    > d_N.
\end{equation*}
Then, by Lemma \ref{l:bK}\ref{l:bK:b}, $\bL$ is a refinement of the partition $\bK_\zeta$ associated to $(X_\zeta, d^{\bK}(X_{\zeta}))$. In addition, by Lemma \ref{l:bK}\ref{l:bK:e}, $\bK_\zeta$ is a refinement of $\bK$. To show that the refinement $\bL$ of $\bK$ is strict, let $K \in \bK$ be such that $\zeta = \eta_K$. Then, $|K| \geq 2$, and Proposition \ref{p:cluster:split} applied with $X_{\zeta}^K$ guarantees that there exists a partition $\{K^1, K^2\}$ of $K$ such that  
  \begin{equation*} 
    \min_{i \in K^1} \min_{j \in K^2} |X_{\zeta}^i - X_{\zeta}^j|
    \geq \sqrt{ \frac{ 2 R_{\zeta}^{K} }{|K|^3} } \geq \frac17 \sqrt{ \frac2{N^3} }.  
  \end{equation*}
This completes the proof of \ref{pp7}. 

\vspace{0.3cm}

Finally we prove \ref{pp3} by showing that on $A$, $\Dopp(X_\zeta) \le \frac67 + \Dopp(x)$. We argue similarly to the proof of \eqref{pfyb}, but here with opposite inequalities. Let $K, L \in \bK$, $i \in K$ and  $j \in L$ be such that $b^i = -b^j$ and $\Dopp(x) = |x^i - x^j|$. Then, 
    \begin{align*} 
      \Dopp(X_{\zeta}) 
      \leq |X_{\zeta}^i - X_{\zeta}^j|
        &\leq |X_{\zeta}^i - M_{\zeta}^K| + |M_{\zeta}^K - M_{\zeta}^{L}| + |M_{\zeta}^{L} - X_{\zeta}^j| \\
        &\leq |M_\zeta^K - M^K(x)| + |M^K(x) - M^{L}(x)| + |M^L(x) - M_\zeta^{L}| + \frac27 \\
        &\leq |M^K(x) - x^i| + |x^i - x^j| + |x^j - M^{L}(x)| + \frac47 \\
        &\leq \Dopp(x) + \frac67.
    \end{align*} 
\end{proof}

\begin{proof}[Proof of Proposition \ref{p:it}]
According to the strong Markov property, we define iteratively the (random) sequences $(\zeta^k)_{k\ge 0}$ of $(\cF_t)_{t\ge 0}$-stopping times and $(\bK^k)_{k\ge 0}$  of $\cF_{\zeta^k}$-measurable partitions of $\ig 1,N\id$. Let $\zeta^0 := 0$ and $\bK^0$ be the partition associated to $(x, \Dopp(x))$ (recall Definition \ref{d:bK}). This implies $d^{\bK^0}(x) = \Dopp(x)$. Combining this with $x\notin \cE_N$, we observe that $\bK^0$ satisfies the conditions in the setting of Lemma \ref{l:it:step}. In the following, we evaluate $(X_t)_{t\ge 0}$ only at $\zeta^k$; we abuse notation and abbreviate $X_k := X_{\zeta^k}$. For all $k \ge 0$:
\begin{itemize}
  \item On the event 
  \[
    B_k 
    := \{ d^{\bK^k}(X_k) > \Dsame(X_k) \} \cap \{ X_k \notin \cE_N \}
    \in \cF_{\zeta^k},
  \]
  Lemma \ref{l:it:step} provides
  a stopping time $\zeta^{k+1} \in [\zeta^k, \tau)$ and an $\cF_{\zeta^{k+1}}$-measurable strict refinement  $\bK^{k+1}$ of $\bK^k$ such that \comm{[below, $B_k$ is st it matches the needed conditions on $x,d$, and $A_k$ st it fits the events that happen wp gtr than $q$]}
  \begin{equation} \label{pfxg}
    \indiq_{B_k} \P(A_{k+1}|\cF_{\zeta^{k}}) \geq q_* \indiq_{B_k}, 
  \end{equation}
  where $q_* = q_*(N,\gamma)>0$ is a deterministic constant and $A_{k+1} \in \cF_{\zeta^{k+1}}$ is
  \begin{equation} \label{pfxf}
    A_{k+1} 
    = \Big\{ d^{\bK^{k+1}}(X_{k+1}) \geq \frac17 \sqrt{  2N^{-3} } d^{\bK^{k}}(X_{k}) \Big\} \bigcap \Big\{\Dopp(X_{k+1}) \leq \frac67 d^{\bK^{k}}(X_{k}) + \Dopp(X_k) \Big\},
  \end{equation}
  which comes from Lemma \ref{l:it:step}-(e),(f). 
  \item On the event $B_k^c$, we simply set $\zeta^{k+1} = \zeta^k$ and $\bK^{k+1} = \bK^{k}$. 
\end{itemize}
It is easy to see that this iterative construction is well-defined. Indeed, $\bK^k$ is for all $k \geq 1$ a refinement of $\bK^0$, and therefore it separates the positive from the negative particles. Also, $\zeta_k < \tau$ for all $k \geq 0$. Hence, for each $k \geq 0$, on $B_k$, the couple $(X_k, \bK^k)$ satisfies the assumptions in Lemma \ref{l:it:step}.

Let for all $\omega\in \Omega$, $L(\omega) := \inf \{ k \ge 0 : \omega \in B_k^c\}$. Since the sequence $(B_k)_{k\ge 0}$ is nonincreasing in the sense of inclusion, $\{  k \le L-1\} = B_k$ and $\{  k \ge L\} = B_k^c$. Since $|\bK^0| \geq 2$, and $|\bK^{k+1}| \geq |\bK^k| + 1$ on $B_k$, we have $\P(B_{N-2})=0$ and thus $L \leq N-2$ a.s. We set $\eta := \zeta_L$, and note that $\eta \le \zeta_{N-2} < \tau$.

Next we show that a.s. $\eta > 0$. We reason on the event $\{ \eta = 0 \}$. Then, we have $\zeta_k = 0$ for all $k \ge 0$, thus $X_k = x \notin \cE_N$ for all $k \ge 0$ and $L = \inf \{ k\ge 0 : d^{\bK^k}(x)\le \Dsame(x)\}$. Then, according to Lemma \ref{l:it:step}\ref{pp6} and $L \leq N-2$, we get
\begin{align*}
  \Dsame(x) 
  \ge d^{\bK^L}(x)
  \geq \Big( \frac17 \sqrt{  2N^{-3} } \Big)^{N-2} d^{\bK^0}(x)
  \geq c_N d^{\bK^0}(x)
  = c_N \Dopp(x),
\end{align*}
which contradicts with $x \notin \cE_N$. Then $\P(\eta = 0) = 0$.

\vspace{0.3cm}

Last we prove $\P(X_\eta \in \cE_N) \geq q$ for some $q = q(\gamma, N) > 0$. We consider the event \comm{[In the below, note that the $\in \cF_{\zeta^{k+1}}$ follows from $L$ being a discrete $(\cF_{\zeta^{k}})_k$-stopping time]} 
$$
D:=\cap_{k\ge 1} C_k, \qquad 
C_k := (B_{k-1} \cap A_k) \cup B_{k-1}^c \in \cF_{\zeta^k}.
$$
First, we construct $q > 0$ and show that $\P(D) \geq q$. For $k\ge N-2$, we have $L \leq k$. Then, $\P(B_k^c) = 1$ and thus $\P(C_{k+1}) = 1$. Therefore, $\P(D) = \P(D_{N-2})$, where $D_\ell := \cap_{k=1}^\ell C_k \in \cF_{\zeta^{\ell}}$ for all $1 \le \ell \le N-2$.
We get for all $1 \leq \ell \le N-3$ by the strong Markov property and \eqref{pfxg}
\begin{align*}
  \P(D_{\ell+1})
  & =  \P(A_{\ell+1} \cap B_\ell \cap D_\ell) + \P( (B_\ell)^c \cap D_\ell) \\
  &= \E [ \indiq_{D_\ell} \indiq_{B_\ell} \P(A_{\ell+1}|\cF_{\zeta_\ell} ) ]+\P( (B_\ell)^c \cap D_\ell) \\
  &\geq q_* \P ( B_\ell \cap D_\ell ) + \P( (B_\ell)^c \cap D_\ell) \\
  &\ge q_* \P(D_\ell).
\end{align*}
We conclude that $\P(D) \ge q_*^{N-3}\P(A_1)\ge q_*^{N-2} =: q$.

Second, we show that $D \subset \{X_\eta \in \cE_N\}$. We reason on $D$. By definition of $L$, we either have $X_\eta \in \cE_N$, in which case we are done, or $X_\eta\notin \cE_N$ and 
$d^{\bK^L}(X_L) \leq \Dsame(X_L)$. 
In the latter case we set
\begin{equation*}
  g_k 
     := \frac{ \Dopp( X_k ) }{d^{\bK^k} ( X_k )}
     \quad \mbox{for all } k\ge 0.
\end{equation*}
Note that $g_0 = 1$. By definition of $D$ and $L$, we obtain from \eqref{pfxf} that for all $k\in \ig 0,L-1\id$, (setting $\alpha := \frac17 \sqrt{  2N^{-3} } $)
  \begin{align*}
     g_{k+1} 
     = \frac{ \Dopp( X_{k+1} ) }{d^{\bK^{k+1}} ( X_{k+1} )}
     \leq \frac{ \frac67 d^{\bK^k} (X_k) + \Dopp( X_k ) }{ \alpha d^{\bK^k} ( X_k ) } 
     = \frac1\alpha \Big( g_k + \frac67 \Big).
   \end{align*} 
   Solving with equality yields 
   \begin{equation*}
     g_k \leq A \alpha^{-k} + B, \qquad A = 1 + \frac{6}{7(1 - \alpha)}, \quad B = \frac{-6}{7(1 - \alpha)}.
   \end{equation*}
   Recalling $d^{\bK^L}(X_L) \leq \Dsame(X_L) = \Dsame(X_\eta)$ and $L\le N-2$, we get
   \begin{equation*}
     \frac{\Dopp (X_\eta)}{\Dsame (X_\eta)}
     \leq \frac{\Dopp(X_L)}{d^{\bK^L}(X_L)}
     = g_{L} \leq A \alpha^{-N+2} + B.
   \end{equation*}
   Applying $B \leq 0$ and $A \leq 7$, we obtain
   \begin{equation*}
     \frac{\Dopp (X_\eta)}{\Dsame (X_\eta)}
     \leq \frac{ 7^{N-1} }{ 2^{\frac12 (N - 2)} } N^{\frac32 (N-2)} < \frac1{c_N}.
   \end{equation*}
\end{proof}

\section{Global well-posedness}
\label{s:mainT}

In this section, we give a rigorous meaning to the SDE system \eqref{SDE}. In particular, in Definition \ref{d:SDE} below we formalize the collision rule from Section \ref{s:intro:SC} given by particle removal and define the global solution. In Theorem \ref{t:SDE} below we formalize Properties \ref{thma}--\ref{thme} from Section \ref{s:intro:main:results}. We prove Theorem \ref{t:SDE} by iteratively applying Lemma \ref{l:exun:X}, Proposition \ref{p:cty:tau} and Theorem \ref{t:3plus} over the collision times. 
 
To state Definition \ref{d:SDE}, we add an abstract cemetery point $\triangle$ to $\R^2$ (as the state space of the position of a particle) to put the particle at upon removal. Since we need a concept of continuity for particle trajectories on $\R^2\cup \{\triangle\}$, we need to introduce a topology on it. By interpreting $\R^2\cup \{\triangle\}$ as embedded in $\R^3$ as $(\R^2 \times \{0\}) \cup \{ (0,0,1) \}$, we take the topology as inherited from the Euclidean distance in $\R^3$. Let 
$$
\cE_\triangle := \big\{ x\in (\R^2\cup \{\triangle\})^N : \mbox{ for all } i,j : x^i\in (\R^2\setminus \{x^j \} )\cup \triangle\big\}
$$
be the set of configurations in which all remaining particles have distinct positions. We equip it with the product topology.
For all $x\in \cE_\triangle$, let 
$$
  \cN(x) :=  \{i\in \ig 1,N\id : x^i \ne \triangle\}
$$ 
be the index set of the remaining particles.

\begin{defn}[Solution of {\eqref{SDE}}] \label{d:SDE} 
Let $\gamma > 0$, $N \geq 2$, $b \in \{-1,1\}^N$ and $(\Omega,\cF,(\cF_t)_{t\ge 0},\P)$ be a filtered probability space on which is defined a $2N$-dimensional standard Brownian motion $(B^1_t,\dots,B^N_t)_{t\ge 0}$ and an $\cF_0$-measurable r.v.\ $X_0 \in \cE$. A c\`adl\`ag, $(\cF_t)_{t\ge 0}$-adapted, $\cE_\triangle$-valued process $(X_t)_{t\ge 0}$ is a solution of \eqref{SDE} if it satisfies the following:
\begin{itemize}
\item denoting for all $x\in \cE_\triangle$ by $\P_x$ the law of $(X_t)_{t\ge 0}$ starting from $x$, we have that 
$$
(\Omega,\cF,(\cF_t)_{t\ge 0},\P,(X_t)_{t\ge 0},(\P_x)_{x\in \cE_\triangle})
$$
is a strong Markov process,
    \item denoting $\cN_t := \cN(X_t)$ for all $t\ge 0$, $(\cN_t)_{t\ge 0}$ is a.s.\ nonincreasing for the inclusion, 
    \item define iteratively the stopping times $\tau_0 := 0$ and
    $$
    \tau_{k+1} := \inf\{ t \ge \tau_k : |\cN_t| < |\cN_{\tau_k}| \} \qquad \text{for all } k \in \N.
    $$
    Let $m$ be the (random) integer such that $\tau_m < \tau_{m+1} = \infty$.
    We have that: 
    \begin{enumerate}
        \item for all $k \ge 0$ with $\tau_k < \infty$ and all $t\in [\tau_k,\tau_{k+1})$, 
    $$
    \mbox{ for all } i\in \cN_{\tau_k}, \quad X^i_t = X^i_{\tau_k-} + B^i_t - B^i_{\tau_k} + \gamma \int_{\tau_k}^t\sum_{j\in \cN_{\tau_k}} b^ib^j f(X^i_s - X^j_s)\dd s,
    $$
    where $X_{0-} := X_0$,
        \item for all $k\in \ig 1,m \id$ there a.s.\ exist $L \in \N$ and disjoint $ \{ i_\ell, j_\ell \} \subset \cN_{\tau_{k-1}}$ for $\ell \in \bbrack{1,L}$ such that
        $\cN_{\tau_{k-1}} \setminus \cN_{\tau_{k}} = \cup_{\ell = 1}^L \{ i_\ell, j_\ell \}$ and for all $\ell \in \bbrack{1,L}$ we have $X_{\tau_{k}-}^{i_\ell} = X_{\tau_{k}-}^{j_\ell}$ and $b^{i_\ell} = -b^{j_\ell}$.
    \end{enumerate}     
\end{itemize}
\end{defn}

We make several remarks to parse Definition \ref{d:SDE}. The monotonicity of $\cN_t$ implies that once particles are removed, they can never return back to the system. Since $X_t$ is c\`adl\`ag, $\cN_t$ is c\`adl\`ag too. The stopping times $\tau_k$ mark, when positive and finite, the times at which some particles are removed. $m$ is the number of times such removals happen, which is obviously bounded from above by $N$. In point 1, by construction we have that $\{ [\tau_k,\tau_{k+1}) \}_{k=0}^m$ is a tessellation of $[0,\infty)$ of nonempty intervals; hence the SDE gives, through the strong Markov property, a rigorous sense for the SDE in \eqref{SDE}. Note that the SDE covers the special cases $|\cN_t| \in \{0,1\}$, and that the term `$X^i_{\tau_k-}$' guarantees the continuity of the trajectories at $\tau_k$ of the remaining particles. Finally, point 2 specifies that particles can only be removed if the conditions of the removal rule stated in Section \ref{s:intro:SC} are met. Then, the particle separation condition given by $X_t \in \cE_\triangle$ forces that instantly particle pairs are removed until no two particles are at the same position in $\R^2$. This formalizes the `collision rule' mentioned in \eqref{SDE}.

\begin{thm}[Well-posedness and properties] \label{t:SDE}
Let $\gamma$, $N$, $b$, $(\Omega,\cF,(\cF_t)_{t\ge 0},\P)$, $(B^1_t,\dots,B^N_t)_{t\ge 0}$ and $X_0$ be as in Definition \ref{d:SDE}. Then, there exists a unique (up to indistinguishability) solution $(X_t)_{t\ge 0}$ of \eqref{SDE}. Moreover, with $\cN_t, \tau_k, m$ defined from $(X_t)_{t\ge 0}$ as in Definition \ref{d:SDE}, $(X_t)_{t\ge 0}$ satisfies:
\begin{enumerate}[label=(\roman*)]
    \item $\displaystyle m = \frac12 \Big( |\cN_0| - \Big|   \sum_{i \in \cN_0} b^i \Big| \Big)$, and
    \item $|\cN_{\tau_{k-1}} \setminus \cN_{\tau_{k}}| = 2$ for all $k \in \bbrack{1,m}$. 
\end{enumerate}
\end{thm}

\begin{proof}[Proof of Theorem \ref{t:SDE}]
First we prove the existence. We will construct a solution $(X_t)_{t \ge 0}$ iteratively. By Lemma \ref{l:exun:X} and Proposition \ref{p:cty:tau} there exists a $\SC(b, X_0)$-process $(\tilde X_t)_{t \ge 0}$ driven by $(B_t)_{t \ge 0}$ with lifetime $\tilde \tau$ in the sense of Remark \ref{r:SC}. Since $\tilde X_t \in \cE$ for all $t < \tilde \tau$, we have $\cN(X_t) = \cN(X_0)$ for all $t < \tilde \tau$, and thus Lemma \ref{l:exun:X} and Proposition \ref{p:cty:tau} demonstrate that $(\tilde X_t)_{t \ge 0}$ satisfies all conditions in Definition \ref{d:SDE} on $[0,\tilde \tau)$. We set $X_t := \tilde X_t$ for all $t < \tilde \tau$. On the event $\{ \tilde \tau = \infty\}$, Theorem \ref{t:3plus} implies that $|\sum_{i=1}^N b^i | = N$, which completes the construction of $(X_t)_{t \ge 0}$. On the event $\{\tilde  \tau < \infty\}$, we have $ X_{\tilde \tau-} \notin \cE$. Then, Theorem \ref{t:3plus} guarantees that $|\sum_{i=1}^N b^i | < N$ and $ X_{\tilde \tau-} \in \cG_2$. In particular, the condition in point 2 of Definition \ref{d:SDE} is met with $\tau_1 = \tilde \tau$ and with a single pair $\{i,j\}$. Setting $X_t^i := X_t^i := \triangle$ for all $t \geq \tau_1$ and $X_{\tau_1}^k := X_{\tau_1-}^k$ for all $k \in \cN_{\tau_1}$, we observe that $( X_t)_{t \ge 0}$ satisfies all conditions in Definition \ref{d:SDE} on $[0, \tau_1]$. This completes the iteration step. For the sake of clarity, let us mention that the next iteration step continues the construction of the remaining part $((X_t^k)_{k \in \cN_{\tau_1}})_{t \in [\tau_1, \infty)}$ by considering the $\SC((b^k)_{k \in \cN_\tau}, (X_{\tau_1-}^k)_{k \in \cN_{\tau_1}})$-process $(\hat X_t)_{t \ge 0}$ driven by $((B_t^k - B_\tau^k)_{k \in \cN_{\tau_1}})_{t \ge 0}$. It is easy to verify that this iteration finishes in at most $N$ steps, and that the resulting process $(X_t)_{t \ge 0}$ satisfies the conditions in Definition \ref{d:SDE} and the two further conditions from Theorem \ref{t:SDE}.

Next we prove uniqueness. Let $(X_t)_{t \ge 0}$ be a solution with corresponding $\cN_t, \tau_k, m$ from  Definition \ref{d:SDE}. Then, $(X_t)_{t \ge 0}$ is a $\SC(b, X_0)$-process on $[0,\tau_1)$, and thus it is uniquely defined up to $\tau_1$. By point 2 of Definition \ref{d:SDE}, $X_{\tau_1-} \notin \cE$, and thus $\tau_1$ is the (uniquely defined) lifetime of the $\SC(b, X_0)$-process. Then, by Theorem \ref{t:3plus} and point 2, the two particles that jump at $\tau_1$ to $\triangle$ are uniquely determined. By point 1, the remaining particles continue to evolve from $\tau_1$ onwards by an $\SC$-process with data similar to the existence proof. Iterating this construction over the collision times $\tau_k$, we obtain the uniqueness of $(X_t)_{t \ge 0}$.
\end{proof}

\section*{Acknowledgements}

PvM has received financial support from JSPS KAKENHI Grant Number JP24K06843.


\begin{thebibliography}{vMMO26}

\bibitem[AMS11]{AmbrosioMaininiSerfaty11}
L.~Ambrosio, E.~Mainini, and S.~Serfaty.
\newblock Gradient flow of the {C}hapman--{R}ubinstein--{S}chatzman model for
  signed vortices.
\newblock In {\em Annales de l'Institut Henri Poincare (C) Non Linear
  Analysis}, volume 28(2), pages 217--246, 2011.

\bibitem[BOS07]{BethuelOrlandiSmets07}
F.~Bethuel, G.~Orlandi, and D.~Smets.
\newblock Dynamics of multiple degree {G}inzburg-{L}andau vortices.
\newblock {\em Communications in Mathematical Physics}, 272:229--261, 2007.

\bibitem[BS24]{BoursierSerfaty24ArXiv}
J.~Boursier and S.~Serfaty.
\newblock Dipole formation in the two-component plasma.
\newblock {\em arXiv:2410.01025}, 2024.

\bibitem[BS25]{BoursierSerfaty25ArXiv}
J.~Boursier and S.~Serfaty.
\newblock Multipole and {B}erezinskii--{K}osterlitz--{T}houless transitions in
  the two-component plasma.
\newblock {\em arXiv:2509.09449}, 2025.

\bibitem[CP16]{CattiauxPedeches16}
P.~Cattiaux and L.~P\'ed\`eches.
\newblock {The 2-D stochastic Keller-Segel particle model: existence and
  uniqueness}.
\newblock {\em Latin American Journal of Probability and Mathematical
  Statistics}, 13:447--463, 2016.

\bibitem[FJ17]{FournierJourdain17}
N.~Fournier and B.~Jourdain.
\newblock Stochastic particle approximation of the {K}eller--{S}egel equation
  and two-dimensional generalization of {B}essel processes.
\newblock {\em Annals of Applied Probability}, 27(5):2807--2861, 2017.

\bibitem[FT25]{FournierTardy25}
N.~Fournier and Y.~Tardy.
\newblock Collisions of the supercritical {K}eller--{S}egel particle system.
\newblock {\em Journal of the European Mathematical Society (EMS Publishing)},
  27(10), 2025.

\bibitem[KS98]{KaratzasShreve98}
I.~Karatzas and S.~E. Shreve.
\newblock {\em Brownian Motion and Stochastic Calculus}.
\newblock Springer, New York, 2nd edition, 1998.

\bibitem[Lew22]{Lewin22}
M.~Lewin.
\newblock Coulomb and {R}iesz gases: {T}he known and the unknown.
\newblock {\em Journal of Mathematical Physics}, 63(6), 2022.

\bibitem[LSZ17]{LebleSerfatyZeitouni17}
T.~Lebl{\'e}, S.~Serfaty, and O.~Zeitouni.
\newblock Large deviations for the two-dimensional two-component plasma.
\newblock {\em Communications in Mathematical Physics}, 350(1):301--360, 2017.

\bibitem[LY16]{LiuYang16}
J.-G. Liu and R.~Yang.
\newblock Propagation of chaos for large {B}rownian particle system with
  {C}oulomb interaction.
\newblock {\em Research in the Mathematical Sciences}, 3:1--33, 2016.

\bibitem[MZ05]{MasmoudiZhang05}
N.~Masmoudi and P.~Zhang.
\newblock Global solutions to vortex density equations arising from
  sup-conductivity.
\newblock {\em Annales de l'Institut Henri Poincar\'e}, 22(4):441--458, 2005.

\bibitem[RY99]{RevuzYor99}
D.~Revuz and M.~Yor.
\newblock {\em Continuous Martingales and Brownian Motion}.
\newblock Springer-Verlag Berlin Heidelberg, 1999.

\bibitem[SBO07]{SmetsBethuelOrlandi07}
D.~Smets, F.~Bethuel, and G.~Orlandi.
\newblock Quantization and motion law for {G}inzburg--{L}andau vortices.
\newblock {\em Archive for Rational Mechanics and Analysis}, 183(2):315--370,
  2007.

\bibitem[Tar24]{Tardy24}
Y.~Tardy.
\newblock Weak convergence of the empirical measure for the {K}eller--{S}egel
  model in both subcritical and critical cases.
\newblock {\em Electronic Journal of Probability}, 29:1--35, 2024.

\bibitem[vMMO26]{VanMeursMurokawaOshikawaXX}
P.~van Meurs, S.~Murokawa, and I.~Oshikawa.
\newblock Ill-posedness of discrete screw dislocation dynamics.
\newblock {\em In preparation}, 2026.

\bibitem[vMPP22]{VanMeursPeletierPozar22}
P.~{v}an Meurs, M.~A. Peletier, and N.~Po{\v{z}}{\'a}r.
\newblock Discrete-to-continuum convergence of charged particles in 1{D} with
  annihilation.
\newblock {\em Archive for Rational Mechanics and Analysis}, 246(1):241--297,
  2022.

\bibitem[vMPS25]{VanMeursPeletierSlangen25}
P.~van Meurs, M.~A. Peletier, and T.~Slangen.
\newblock Global existence and mean--field limit for a stochastic interacting
  particle system of signed {C}oulomb charges.
\newblock {\em Potential Analysis}, 63(4):1699--1733, 2025.

\end{thebibliography}
\end{document}